\documentclass[12pt]{amsart}

\usepackage{amsthm,amsmath,amsfonts,amssymb}

\theoremstyle{plain}
\newtheorem{thm}{Theorem}[section]
\newtheorem{lemma}[thm]{Lemma}
\newtheorem{cor}[thm]{Corollary}
\newtheorem{prop}[thm]{Proposition}

\newtheorem{CorThm}{Theorem}

\newtheorem*{thm*}{Theorem}
\newtheorem*{prop*}{Proposition}
\newtheorem*{lemma*}{Lemma}

\theoremstyle{remark}
\newtheorem{rmk}[thm]{Remark}
\newtheorem{example}[thm]{Example}

\theoremstyle{definition}
\newtheorem{defn}[thm]{Definition}

\numberwithin{equation}{section}

\def\E{\mathbb{E}}

\def\R{\mathbb{R}}
\def\N{\mathbb{N}}

\def\cJ{\mathcal{J}}

\def\cX{\mathcal{X}}

\DeclareMathOperator{\Id}{Id}

\DeclareMathOperator{\proj}{proj}

\DeclareMathOperator*{\argmin}{argmin}

\newenvironment{PfofEntropyEquality}[1]
{\par\vskip2\parsep\noindent{\sc Proof of Theorem\ \ref{Thm:EntropyEquality}. }}{{\hfill
$\Box$}
\par\vskip2\parsep}

\newenvironment{PfofThmFurstenberg}[1]
{\par\vskip2\parsep\noindent{\sc Proof of Theorem\ \ref{Thm:Furstenberg}. }}{{\hfill
$\Box$}
\par\vskip2\parsep}

\newenvironment{PfofThmGenObsModelConsistency}[1]
{\par\vskip2\parsep\noindent{\sc Proof of Theorem\ \ref{Thm:GenObsModelConsistency}. }}{{\hfill
$\Box$}
\par\vskip2\parsep}

\newenvironment{PfofThm_PropSPEC}[1]
{\par\vskip2\parsep\noindent{\sc Proof of Theorem\ \ref{Thm:SPEC}. }}{{\hfill
$\Box$}
\par\vskip2\parsep}

\newenvironment{PfofThm_LS}[1]
{\par\vskip2\parsep\noindent{\sc Proof of Theorem\ \ref{Thm:LSConsistency}. }}{{\hfill
$\Box$}
\par\vskip2\parsep}

\newenvironment{PfofCorRd}[1]
{\par\vskip2\parsep\noindent{\sc Proof of Corollary\ \ref{Cor:Rd}. }}{{\hfill
$\Box$}
\par\vskip2\parsep}

\newenvironment{PfofPropNegative}[1]
{\par\vskip2\parsep\noindent{\sc Proof of Proposition\ \ref{Prop:Negative}. }}{{\hfill
$\Box$}
\par\vskip2\parsep}

\title[Risk minimization and complexity of dynamical models]{Empirical risk minimization and complexity of dynamical models}

\begin{document}

\author{Kevin McGoff and Andrew B. Nobel}
\address{Kevin McGoff\\
Department of Mathematics\\
University of North Carolina at Charlotte\\
Charlotte, NC 28223}
\email{kmcgoff1@uncc.edu}
\urladdr{https://clas-math.uncc.edu/kevin-mcgoff/}
\address{Andrew B. Nobel\\
Department of Statistics and Operations Research\\
University of North Carolina at Chapel Hill\\
308 Hanes Building\\
Chapel Hill, NC 27599}
\email{nobel@email.unc.edu}
\thanks{KM acknowledges the support of NSF grant DMS-1613261. AN acknowledges the support of NSF grants DMS-1613072, DMS-1310002, and DMS-1613261.}

%
%
%
%
%
%

\begin{abstract}
A dynamical model consists of a continuous self-map $T: \cX \to \cX$ of a compact state space $\cX$ 
and a continuous observation function $f: \cX \to \R$.  
This paper considers the fitting of a parametrized family of dynamical models to an observed real-valued 
stochastic process using empirical risk minimization.
The limiting behavior of the minimum risk parameters is studied in a general setting.
We establish a general convergence theorem for minimum risk estimators and ergodic observations.
We then study conditions under which empirical risk minimization can effectively
separate the signal from the noise in an additive observational noise model.  
The key, necessary condition in the latter results is that the family
of dynamical models has limited complexity, which is quantified through a notion of entropy for 
families of infinite sequences.  
Close connections between entropy and limiting average
mean widths for stationary processes are established.
\end{abstract}

\renewcommand{\subjclassname}{MSC 2010}
%
%


\maketitle

\section{Introduction} \label{Sect:Intro}

Empirical risk minimization is a common approach to model fitting and estimation in a variety of parametric and non-parametric problems.  
In this paper we investigate the use of empirical risk minimization to fit a family of dynamical models to an observed stochastic process.
Formally, a dynamical model consists of a continuous transformation $T : \cX \to \cX$ on a compact metric space $\cX$,
and a continuous observation function $f : \cX \to \R$.  
Let $T^k$ denote the $k$-fold composition of $T$ with itself, and let $T^0$ 
be the identity map on $\cX$. 
From each initial state
$x \in \cX$ the dynamical model $(T,f)$ yields a real-valued sequence
\[
\bigl\{ f(T^k x) \, = \, f \circ T^k (x) : k \geq 0 \bigr\} \, \subseteq \, \R, 
\]
obtained by applying the observation function $f$ to a deterministic sequence of states
generated by repeated iteration of the transformation $T$.  In general, $f$ need 
not be injective, so one cannot necessarily recover the underlying states from the sequence $\{f(T^k x)\}_k$.

In what follows we consider an indexed family $\mathcal{D} = \{ (T_\theta, f_\theta) : \theta \in \Theta \}$
of dynamical models defined on a common compact metric space $\cX$ satisfying the following conditions:
\begin{enumerate}

\item[(D1)]
the index set $\Theta$ is a compact metric space;

\item[(D2)]
the map $(\theta,x) \mapsto T_{\theta}(x)$  from $\Theta \times \cX$ to $\cX$ is continuous;

\item[(D3)]
the map  $(\theta,x) \mapsto f_{\theta}(x)$ from $\Theta \times \cX$ to $\R$ is continuous.

\end{enumerate}
Condition (D2) ensures that each transformation $T_\theta$ is continuous and that the action
of $T_\theta$ is continuous in $\theta$.
Condition (D3) ensures that each observation function $f_\theta$ is continuous and that observations
vary continuously with $\theta$.  In particular, there exists a constant $K_\mathcal{D} > 0$ 
such that $|f_\theta(x)| \leq K_\mathcal{D}$
for every $x \in \cX$ and $\theta \in \Theta$.
Examples of families of systems satisfying these conditions are given in Section \ref{Sect:PositiveExamples}.

By definition, dynamical models are deterministic, as the sequence of observations generated by a model is fully determined once the initial condition is given.  
However, as noted in Section \ref{Sect:AssociatedProcesses}, each dynamical model has a set of invariant measures and these measures give rise to a family of stationary processes.
In the large sample limit, fitting a family of dynamical models leads directly to a variational problem 
involving its associated processes.  Conditions (D1)-(D3) ensure that the set of associated processes 
is non-empty and that the limiting variational problem is well-defined. 
We make no explicit assumptions about the invariant measures of any individual model.

In the context of this paper, dynamical models represent low order regularities of potential interest, 
such as periodicity, multi-periodicity, constrained growth behavior, and hierarchical structure.
Fitting a family of such models to an observed stochastic process is a means of identifying select
underlying regularities in the observed process, which may be subject to noise.
While the observed process is likely to be complex, our primary focus is on model families $\mathcal{D}$ 
having limited complexity, quantified through the condition that the entropy $h(\mathcal{D})$ of the
family is zero.  The entropy of a family of models is defined in Section \ref{Sect:EntropyDef}.

\subsection{Minimum Risk Fitting of Dynamical Models}

Let $\mathcal{D}$ be a family of dynamical models that capture some behavior of interest, and
let $\mathbf{Y} = Y_0, Y_1, \ldots \in \R$ be an observed stationary ergodic process.  
Suppose that we wish to identify regularities in $\mathbf{Y}$ by fitting the observed values 
of the process with models in $\mathcal{D}$.  
We do not assume that the observed process $\mathbf{Y}$ is 
generated by a process in $\mathcal{D}$. 
Let $\ell: \R \times \R \to \R$ be a nonnegative loss function 
that is jointly lower semicontinuous in its arguments. We require the following integrability condition:
\begin{equation} \label{eq:loss-integ}
\tag{C1} \E \biggl[ \sup_{|u| \leq K_\mathcal{D}} \ell \bigl( u,Y_0 \bigr) \biggr] \ < \ \infty . 
\end{equation}
(If the supremum in (\ref{eq:loss-integ}) 
is not measurable, then one may replace the
expectation by an outer expectation.)  
For each $n \geq 0$, $\theta \in \Theta$, and $x \in \cX$ define 
\[
R_n (\theta : x) \ = \ 
\frac{1}{n} \, \sum_{k = 0}^{n-1} \ell( f_{\theta} \circ T_{\theta}^k(x),Y_k \bigr)
\]
to be the empirical risk of the model $(T_{\theta}, f_{\theta})$ with initial state $x$ relative to the
first $n$ observations of $\mathbf{Y}$.
We formalize empirical risk minimization as follows.

\vskip.1in

\begin{defn}
A sequence of measurable functions $\theta_n: \R^n \to \Theta$, $n \geq 1$, will be called 
(empirical) minimum risk estimates for $\mathcal{D}$ if 
\begin{equation}
\label{LSES}
\lim_n \, \inf_{x} R_n (\hat{\theta}_n : x) \ = \ \lim_n \, \inf_{\theta} \, \inf_{x} R_n (\theta : x)
\ \ \ \mbox{w.p.1},
\end{equation}
where $\hat{\theta}_n : =\theta_n(Y_0,\ldots, Y_{n-1})$.
\end{defn}

\begin{rmk}
Existence of the limit on the right hand side of (\ref{LSES}) follows from Kingman's subadditive 
ergodic theorem (under (\ref{eq:loss-integ})). 
Existence of the limit on the left hand side of (\ref{LSES}) is part of the definition.  
\end{rmk}

\begin{rmk}
The definition of minimum risk estimates is a way to formalize empirical risk minimization in the context of fitting dynamical models. Note that the definition does not require exact minimization for each $n$: it only requires that the average loss is asymptotically minimized by the sequence of estimates. Thus, our results apply to any sequence of approximate minimizers satisfying this condition. This generalization is important, as approximate versions of these procedures are used in practice when fitting dynamical systems to observations.
\end{rmk}

The primary goal of this paper is to investigate and characterize the limiting behavior of minimum risk estimates. 
Three main results are presented, corresponding to three levels of generality.
At the highest level, 
we provide a variational characterization of the limiting behavior of minimum risk estimates 
(Theorem \ref{Thm:GEN}).
We then focus on a natural signal plus noise setting and show that if the family of dynamical models 
has low complexity then empirical risk minimization effectively separates the signal and the noise 
(Theorem \ref{Thm:GenObsModelConsistency}).
In the low complexity setting, we show that 
if the signal arises from a model in the family and the noise is appropriately centered with 
respect to the loss, then the minimum risk estimates are consistent (Theorem \ref{Thm:SPEC}).
A negative result (Proposition \ref{Prop:Negative}) shows that empirical risk minimization can 
be inconsistent when the family of dynamical models has high complexity.

Beyond the results above, the main contribution of our work is a systematic 
treatment of identifiability and complexity for families
of dynamical models, with a focus on the misspecified case in which the observed process is 
not generated by a model in the family.  We avoid
assumptions on the invariant measures of the models in $\mathcal{D}$ by working with the 
processes they generate.  In particular,
the limit set of empirical risk estimators is characterized in terms of a distortion-based 
projection of the observed process onto the family of processes associated with $\mathcal{D}$.
In addition, we introduce an entropy-based definition of complexity for families of dynamical models 
and families of infinite sequences 
that may be of independent interest.
Our notion of entropy has close connections with topological entropy, studied in dynamical systems, and 
with stochastic mean widths, studied in empirical process theory.

Both the statements and proofs of our results rely on the concept of joinings, which are stationary couplings of stochastic processes. 
Joinings, introduced by Furstenberg \cite{Furstenberg1967}, have been well-studied in ergodic theory, but have not been 
widely applied to problems of statistical inference. 
Our results show that joinings are intimately connected with minimum risk fitting of dynamical models.
Several tools from the theory of joinings, including disjointness and relatively independent joinings, 
play an important role in our analysis.

\section{Definitions and first results}

In this section we introduce some concepts and notation that will be useful in what follows. We also state our first convergence result, Theorem \ref{Thm:GEN}.

\subsection{Processes associated with dynamical models} \label{Sect:AssociatedProcesses}

Let $(T,f)$ be a dynamical model on a compact metrizable state space $\cX$.
Recall that a Borel probability measure $\mu$ on $\cX$ is said to be 
\textit{invariant} under $T$ if $\mu(T^{-1}A) = \mu(A)$ for all Borel 
sets $A \subseteq \cX$. 
Let $\mathcal{M}(\cX,T)$ be the set of Borel measures on $\cX$ that are 
invariant under $T$, which is nonempty (see \cite[p.152]{Walters1982}). 
To each measure $\mu \in \mathcal{M}(\cX,T)$ there is an associated 
real-valued process
\[
\mathbf{U} = f(X), f(T X),  f(T^2 X), \ldots 
\] 
where $X \in \cX$ has distribution $\mu$.  The invariance of $\mu$ under $T$ ensures that 
$\mathbf{U}$ is stationary.  Here and in what follows we will regard real-valued processes as 
measures on the infinite product space $\R^\N$ equipped with its Borel sigma-field in the standard
product topology.  

\begin{defn}
Let $\mathcal{D} = \{ (T_\theta,f_\theta) : \theta \in \Theta \}$ be a family of dynamical models.
For each $\theta \in \Theta$ let
\[
\mathcal{Q}_\theta 
\ = \  
\Bigl\{ \mathbf{U} = (f_{\theta} \circ T_{\theta}^k(X))_{k \geq 0} : X \sim \mu \mbox{ with } \mu \in \mathcal{M}(\cX,T_\theta) \Bigr\}
\]
be the set of stationary processes associated with $(T_\theta,f_\theta)$, and let 
$\mathcal{Q}_{\mathcal{D}} = \bigcup_{\theta \in \Theta} \mathcal{Q}_{\theta}$ be the set of processes 
associated with the entire family of models $\mathcal{D}$. 
\end{defn}

\subsection{Joinings and distortion for stationary processes}

The statements and proofs of our principal results rely critically 
on stationary couplings of stationary processes, which are
known as joinings.

\begin{defn} \label{Defn:Joinings}
A \textit{joining} of two stationary processes 
$\mathbf{U} = \{ U_k : k \geq 0 \}$ and 
$\mathbf{V} = \{ V_k : k \geq 0 \}$ 
is a stationary process 
$\mathbf{W} = \{ (\tilde{U}_k, \tilde{V}_k) : k \geq 0 \}$
such that $\tilde{\mathbf{U}}= \{ \tilde{U}_k : k \geq 0 \}$ has the same distribution as $\mathbf{U}$ 
and $\tilde{\mathbf{V}} = \{ \tilde{V}_k : k \geq 0 \}$ has the same distribution as $\mathbf{V}$.
Let $\cJ(\mathbf{U},\mathbf{V})$
denote the family of all joinings of $\mathbf{U}$ and $\mathbf{V}$.
\end{defn}

By definition, a joining of two stationary processes is a coupling of the processes 
that is itself stationary.  
Note that the family $\cJ(\mathbf{U},\mathbf{V})$ always contains the
the so-called independent joining under which $\tilde{\mathbf{U}}$ and $\tilde{\mathbf{V}}$ are independent
copies of $\mathbf{U}$ and $\mathbf{V}$, respectively.
Joinings were introduced by Furstenberg \cite{Furstenberg1967}, and have been widely studied
in ergodic theory \cite{Rue2006,Glasner2003}.  
For notational convenience, we will frequently use  $[\mathbf{U},\mathbf{V}]$ to denote a joining of $\mathbf{U}$ with $\mathbf{V}$.
Also, note that the joining of three or more stationary processes may be defined analogously.

\vskip.1in

\begin{defn}
Let $L: \R \times \R \to \R$ be a nonnegative loss function. 
The $L$-distortion between two stationary processes $\mathbf{U}$ and $\mathbf{V}$ 
is given by
\begin{equation*}
\gamma_L (\mathbf{U},\mathbf{V}) 
\ = \, 
\inf_{\cJ(\mathbf{U},\mathbf{V})} \,
\mathbb{E} \bigl[ L ( U_0, V_0 ) \bigr].
\end{equation*}
\end{defn}

\begin{rmk}
Joinings were used by Ornstein \cite{Ornstein1970,Ornstein1973,Ornstein1974} to define the $\overline{d}$-distance between
finite alphabet stationary processes based on the Hamming metric 
$\mathbb{I}(U_0 \neq V_0)$.  
The $\overline{d}$-distance was extended by Gray et al.\ \cite{Gray1975} to 
stationary processes with general alphabets and to arbitrary metrics 
$\rho(U_0,V_0)$.  The distortion $\gamma_L(\cdot,\cdot)$ is a straightforward
generalization of these distances to nonnegative loss functions $L(\cdot,\cdot)$ that need not be metrics on $\R$.
\end{rmk}

\begin{rmk}
The fact that the infimum defining $\gamma_L(\cdot,\cdot)$ runs over the set of joinings, 
rather than the set of couplings, is critical. 
A minimizing joining makes the average loss between elements of the process as small as 
possible over the entire future. By contrast, a minimizing coupling would make the processes 
as close as possible at time zero, without regard to their behavior in the future.
Moreover, ergodic properties of the processes can severely constrain the set of possible joinings. 
For instance, for many pairs of processes $\mathbf{U}$ and $\mathbf{V}$, the only joining in 
$\mathcal{J}(\mathbf{U},\mathbf{V})$ is the independent joining. Such processes are called disjoint.
\end{rmk}

\subsection{Convergence of Minimum Risk Estimates} \label{Sect:ConvergenceThm}

As noted above, a family $\mathcal{D}$ of dynamical models corresponds to a family 
$\mathcal{Q}_\mathcal{D} = \bigcup_{\theta \in \Theta} \mathcal{Q}_\theta$ of stationary 
processes.  The problem of fitting models in $\mathcal{D}$ to an observed ergodic 
process $\mathbf{Y}$ using the loss $\ell( \cdot, \cdot)$ has a population analog in which we seek
processes in $\mathcal{Q}_\mathcal{D}$ that minimize the distortion $\gamma_\ell(\cdot, \cdot)$
with $\mathbf{Y}$.  The solution to the population problem is the $\gamma_\ell$-projection 
of $\mathbf{Y}$ onto $\mathcal{Q}_\mathcal{D}$, and the corresponding set of parameters
is a natural limit set for empirical risk estimators.
This leads to the following definition, which is given for general loss functions.

\begin{defn}
Let $\mathcal{D}$ be a family of dynamical models parametrized by $\theta \in \Theta$.
Given a nonnegative loss function $L: \R \times \R \to \R$ and a stationary ergodic process $\mathbf{Y}$, 
let
\[
\Theta_L(\mathbf{Y}) 
\ = \ 
\argmin_{\theta \in \Theta} \min_{\mathbf{U} \in \mathcal{Q}_\theta} \gamma_L (\mathbf{U}, \mathbf{Y}),
\] 
which is the set of parameters $\theta$ such that some process in $\mathcal{Q}_\theta$ minimizes the
distortion with $\mathbf{Y}$.

\end{defn}

The proof of the following theorem, 
which relies on results of McGoff and Nobel \cite{McGoffNobel}, is presented in 
Section \ref{Sect:Connections}.

\begin{thm}
\label{Thm:GEN}
Let $\mathcal{D}$ be a family of dynamical models satisfying (D1)-(D3), and let  
$\ell(\cdot,\cdot)$ be a lower semicontinuous loss function.
If $\mathbf{Y}$ is a stationary ergodic process satisfying (\ref{eq:loss-integ}),
then $\Theta_\ell(\mathbf{Y})$ is non-empty and compact and
\vskip.1in
\begin{enumerate}

\item
any sequence $\{ \hat{\theta}_n = \theta_n (Y_0,\ldots,Y_{n-1}) \}$
of minimum risk estimators converges almost surely to $\Theta_\ell(\mathbf{Y})$;

\vskip.1in

\item
for each $\theta \in \Theta_\ell (\mathbf{Y})$, 
there exists a sequence of minimum risk estimators that converges almost surely to $\theta$.

\end{enumerate}
\end{thm}

We emphasize that there is no assumed 
relationship between the observations $\mathbf{Y}$ and the family $\mathcal{D}$.  
Additionally, identifiability of parameters is addressed in a direct way, 
through the distortion $\gamma_\ell$ via the families $\mathcal{Q}_\theta$.

\vskip.1in

Theorem \ref{Thm:GEN} shows that the limiting behavior of minimum risk estimators 
is characterized by the family $\Theta_\ell(\mathbf{Y})$, and in this way the theorem
reduces the asymptotic analysis of empirical risk minimization to the 
analysis of this limit set.  
We show below how analysis of $\Theta_{\ell} (\mathbf{Y})$ yields both positive results 
(e.g. consistency) and negative results (inconsistency) in a signal plus noise setting.

\section{Entropy for sequence families} \label{Sect:EntropyDef}

A key issue in nonparametric inference is how to assess the complexity of a family of models.  
Complexity measures play an important role in establishing consistency, convergence rates, 
and optimality for a variety of inference procedures.  
Although fitting nonlinear dynamical models to ergodic processes is substantially
different from model fitting for classification or regression, complexity still plays an important role in the
consistency of empirical risk minimization.

We assess the complexity of a family $\mathcal{D}$ through the covering numbers of the 
infinite real-valued sequences generated by its constituent models.
From a statistical point of view, it is natural to consider empirical $\ell_2$ covering numbers, while 
from a dynamical systems point of view, it is natural to consider empirical $\ell_\infty$ covering numbers, 
as is done with topological entropy \cite{Walters1982}. As we show below, these two approaches coincide.

Let $\mathbf{u} = (u_k)_{k \geq 0}$ and $\mathbf{v} = (v_k)_{k \geq 0}$ denote infinite sequences  
in $\R^{\N}$.  For each $n \geq 1$ and $1 \leq p \leq \infty$, define pseudo-metrics
$d_{n,p} (\cdot,\cdot)$ as follows:
\[
d_{n,p} (\mathbf{u}, \mathbf{v}) =
\begin{cases}
\left( n^{-1} \sum_{k = 0}^{n-1} |u_k - v_k|^p \right)^{1/p} & \mbox{ if $1 \leq p < \infty$} \\[.1in]
\max_{ 0 \leq k \leq n -1} | u_k - v_k | & \mbox{ if $p = \infty$}.
\end{cases}
\]
Let $\mathcal{U} \subseteq \R^{\N}$ be a family of infinite sequences.  
For each $r > 0$ let $N(\mathcal{U}, r, d_{n,p})$ denote the covering number of the set $\mathcal{U}$ under 
the pseudo-metric $d_{n,p} (\cdot, \cdot)$ at radius $r$.  Let
\begin{equation*}
h_p(\mathcal{U}, r) \, = \, \limsup_n \frac{1}{n} \log N(\mathcal{U}, r, d_{n,p}),
\end{equation*}
which is the exponential growth rate of the covering numbers at radius $r$, and define the $\ell_p$ entropy of
the family $\mathcal{U}$ as the supremum of these growth rates, namely
\begin{equation*}
h_p(\mathcal{U}) \, = \, 
\lim_{r \searrow 0} h_p(\mathcal{U}, r).
\end{equation*}
The following result is established in Section \ref{Sect:EntropyProofs}.  

\vskip.15in

\begin{thm} 
\label{Thm:EntropyEquality}
The $\ell_p$ entropies $h_p(\mathcal{U})$ for $1 \leq p \leq \infty$ are all equal. 
\end{thm}

\vskip.1in

\begin{rmk}
Although it is not needed here, we note that 
Theorem \ref{Thm:EntropyEquality} holds more generally for sets of sequences 
$\mathcal{U} \subseteq A^{\N}$ where $(A, \rho)$ is any metric space such that 
\begin{equation*}
\lim_{r \searrow 0} r \log N(A,r,\rho) = 0
\end{equation*}
and the pseudo metrics $d_{n,p}(\cdot,\cdot)$ are defined in terms of $\rho$.
\end{rmk}

\vskip.1in

\begin{defn}[Entropy of a dynamical family]
The entropy $h(\mathcal{D})$ of a family $\mathcal{D}$ of dynamical models is the common 
value of $h_p(\mathcal{U}_\mathcal{D})$, where
\begin{equation} \label{Eqn:U}
\mathcal{U}_\mathcal{D} = \{ (f_\theta \circ T_\theta^k (x) )_{k \geq 0} : x \in \cX, \, \theta \in \Theta \}
\, \subseteq \, \R^{\N}
\end{equation}
is the set of infinite sequences generated by models in $\mathcal{D}$. 
\end{defn}

\begin{rmk}
It is straightforward to show that $\mathcal{U}_{\mathcal{D}}$ is a compact subset of $\R^{\N}$ in its product topology.
Let $\tau : \R^{\N} \to \R^{\N}$ be the left-shift map defined by $\tau(\mathbf{u})_k = u_{k+1}$ for $k \geq 0$.
Then it is easy to see that $\tau$ is continuous and
$\tau(\mathcal{U}_{\mathcal{D}}) \subset \mathcal{U}_{\mathcal{D}}$.
Thus $(\mathcal{U}_{\mathcal{D}}, \tau)$ is a topological dynamical system that captures
the dynamics of the family $\mathcal{D}$, and the entropy 
$h(\mathcal{D})$ defined above is the topological entropy of this system.
\end{rmk}

We note that the entropy $h(\mathcal{D})$ may also be characterized in terms of the entropies $h(\mathbf{U})$ 
of the processes in $\mathbf{U} \in \mathcal{Q}_{\mathcal{D}}$, which we define in Section \ref{Sect:EntropyAndMeanWidth}.  The following lemma is established in 
Appendix \ref{Sect:AdditionalProofs}.

\begin{lemma} 
\label{Lemma:VarPrin}
For any family $\mathcal{D}$ of dynamical models satisfying (D1)-(D3),
\begin{equation*}
h(\mathcal{D}) = \sup_{\mathbf{U} \in \mathcal{Q}_{\mathcal{D}}} h(\mathbf{U}).
\end{equation*}
\end{lemma}

\section{Signal Plus Noise} \label{Sect:SignalNoise}

We now turn our attention to empirical risk minimization when the observed
process $\mathbf{Y}$ has a signal plus noise structure. 
We assume in what follows that
$Y_k = V_k + \varepsilon_k$ for each $k \geq 0$, where $\mathbf{V} = \{ V_k : k \geq 0 \}$ is a 
stationary ergodic process and $\boldsymbol{\varepsilon} = \{ \varepsilon_k : k \geq 0 \}$
is an i.i.d.\ noise process that is independent of $\mathbf{V}$. We indicate this relationship
using the process-sum notation $\mathbf{Y} = \mathbf{V} + \boldsymbol{\varepsilon}$. 
We require the following integrability conditions:
 \begin{equation}  \label{Eqn:IntCondY}
 \tag{C1} \mathbb{E} \biggl[ \sup_{|x| \leq K_{\mathcal{D}}} \ell(x,Y_0) \biggr] < \infty;
 \end{equation}
 \begin{equation} \label{Eqn:IntCondV}
\tag{C2} \mathbb{E} \biggl[ \sup_{|x| \leq K_{\mathcal{D}}} \ell(x,V_0) \biggr] < \infty;
 \end{equation}
 \begin{equation} \label{Eqn:IntCondEps}
 \tag{C3} \text{for all $u,v \in \R$, \quad  $\mathbb{E} \, \ell(u,v+\varepsilon_0) < \infty$}.
 \end{equation}
Note that (\ref{Eqn:IntCondY}) is the same condition required in the general setting, and (\ref{Eqn:IntCondV}) and (\ref{Eqn:IntCondEps}) refer only to $\mathbf{V}$ and $\boldsymbol{\varepsilon}$, respectively. These conditions involve integrability of the loss with respect to the three processes $\mathbf{Y}$, $\mathbf{V}$, and $\boldsymbol{\varepsilon}$. For example, if $\ell$ is the squared loss and the three processes all have finite second moments, then conditions (C1)-(C3) are satisfied.

Theorem \ref{Thm:GEN} ensures that any sequence of minimum risk estimators will
converge to the set $\Theta_\ell(\mathbf{Y})$ of optimal parameters for $\mathbf{\mathbf{Y}}$.  
Of interest here is when and whether empirical risk minimization can decouple the
signal from the noise and recover the optimal parameters $\Theta_\ell(\mathbf{V})$ for
the signal process $\mathbf{V}$.  We begin with the following general
result.

\vskip.1in

\begin{thm} 
\label{Thm:GenObsModelConsistency}
Let $\mathbf{Y} = \mathbf{V} + \boldsymbol{\varepsilon}$ satisfy (C1)-(C3), and let $\mathcal{D}$ satisfy (D1)-(D3).
If $h(\mathcal{D}) = 0$, then any sequence 
of minimum $\ell$-risk estimates converges almost surely to $\Theta_L(\mathbf{V})$,
where $L(u,v) : = \E \, \ell(u, v + \varepsilon_0)$. 
\end{thm}

\begin{rmk}
Since $\ell(\cdot, \cdot)$ is nonnegative and lower semicontinuous, the auxiliary 
loss function $L(\cdot, \cdot)$ has the same properties (using Fatou's Lemma for the lower semicontinuity). 
\end{rmk}

\begin{rmk}
If a given process $\mathbf{Y}$ can be expressed in two different ways as 
$\mathbf{Y} = \mathbf{V} + \boldsymbol{\varepsilon}$ and 
$\mathbf{Y} = \mathbf{V}'+\boldsymbol{\varepsilon}'$, then the proof of
Theorem \ref{Thm:GenObsModelConsistency} shows that 
$\Theta_L(\mathbf{V}) = \Theta_{L'}(\mathbf{V}')$, where $L'$ is defined using $\boldsymbol{\varepsilon}'$ 
in place of $\boldsymbol{\varepsilon}$.
\end{rmk}

Theorem \ref{Thm:GenObsModelConsistency} shows that if $h(\mathcal{D}) = 0$, then empirical risk 
minimization does indeed decouple the signal from the noise. 
However, the presence of the auxiliary loss function $L(\cdot,\cdot)$ in the conclusion of the theorem 
begs the question of whether the limit set is equal to the limiting parameter set 
$\Theta_{\ell}(\mathbf{V})$ associated with the signal, as one would like. 
The following two results address this question under additional hypotheses.

\vskip.1in

\begin{thm}  \label{Thm:SPEC}
Let $\mathbf{Y} = \mathbf{V} + \boldsymbol{\varepsilon}$ satisfy (C1)-(C3), and let $\mathcal{D}$ be a family of
dynamical models satisfying (D1)-(D3) with $h(\mathcal{D}) = 0$.  Let $\{ \hat{\theta}_n : n \geq 1\}$ be
any sequence of minimum risk estimators based on $\mathbf{Y}$.
\vskip.1in
\begin{enumerate}

\item
If $\ell(u,v) = D_F(v,u)$ is a Bregman divergence and $\E \, \varepsilon_0 = 0$, then  
$\hat{\theta}_n$ converges almost surely to $\Theta_\ell(\mathbf{V})$. 

\vskip.1in

\item
If $\mathbf{V}$ is an ergodic process in $\mathcal{Q}_{\theta_0}$ and 
$\mathbb{E} \, \ell(u,v+\varepsilon_0) \geq \mathbb{E} \, \ell(0,\varepsilon_0)$ for all $u,v$,
with equality if and only if $u = v$, then $\hat{\theta}_n$
converges almost surely to 
$\{ \theta \in \Theta \, : \, \mathbf{V} \in \mathcal{Q}_{\theta} \}$. 

\end{enumerate}

\end{thm}

Theorem \ref{Thm:SPEC} establishes that minimum risk fitting of a zero-entropy dynamical family 
effectively isolates the signal under two types of hypotheses. 
The first places restrictions on the loss function but allows the 
observation process to be quite general. The second places restrictions on the signal and on the joint behavior
of the loss function and the noise processes.

Without the entropy condition $h(\mathcal{D}) = 0$, the conclusion of Theorem \ref{Thm:SPEC} does 
not hold in general, as the following proposition highlights. 
For $\theta \in \Theta$, let $\mathcal{D}_{\theta} = \{ (T_{\theta},f_{\theta}) \}$ be the family consisting of the single dynamical model associated to $\theta$.

\begin{prop} \label{Prop:Negative}
Let $\mathcal{D}$ be a family of dynamical models satisfying (D1)-(D3).  Suppose that 
\begin{itemize}
\item $h(\mathcal{D})>0$;
\item there exists $\theta_0$ such that $h(\mathcal{D}_{\theta_0}) = 0$ and
$\mathcal{Q}_{\theta_0} \setminus \cup_{\theta \neq \theta_0} \mathcal{Q}_{\theta}$
contains an ergodic process $\mathbf{V}$.
\end{itemize}
Then there exists $\sigma_0 > 0$ such that for every i.i.d.\ process 
$\boldsymbol{\varepsilon}$ with $\varepsilon_i \sim N(0,\sigma^2)$ for $\sigma > \sigma_0$,  
the least squares estimates derived from 
$\mathbf{Y} = \mathbf{V} + \boldsymbol{\varepsilon}$ converge to a compact set 
$\Theta_*$ such that $\mathbf{V} \notin \cup_{\theta \in \Theta_*} \mathcal{Q}_{\theta}$. 
Moreover, such a family $\mathcal{D}$ exists. 
\end{prop}

Under the conditions of this proposition, inconsistency holds almost surely despite the fact that the signal process $\mathbf{V}$ is generated by a dynamical model in the family. The idea underlying this phenomenon is that since $\mathcal{D}$ has positive entropy, it is capable of tracking the noise, and then the least squares estimates will overfit the observed sequence. 

\subsection{Squared Loss and Mean Width} 
\label{Sect:SquaredLossMeanWidth}

In this section we give special attention to the squared loss, $\ell(x,y) = (x-y)^2$. 
Analysis of the squared loss naturally leads to another measure of complexity of the family 
$\mathcal{D}$ based on the mean width of sequences generated by $\mathcal{D}$
with respect to the noise.

\begin{defn} \label{Defn:Kappa}
Let $\mathcal{D}$ be a family of dynamical models, and 
let $\boldsymbol{\varepsilon} = (\varepsilon_k)_{k \geq 0}$ be an i.i.d.\ process with mean zero
and finite variance.  
The $n$-sample mean width of $\mathcal{D}$ relative to $\boldsymbol{\varepsilon}$ is 
\begin{equation}
\label{kappan-def}
\kappa_n(\mathcal{D} : \boldsymbol{\varepsilon}) 
\ = \ 
\mathbb{E} \biggl[ \sup_{x,\theta} \sum_{k = 0}^{n-1}  f_{\theta} \circ T_{\theta}^k(x) \cdot \varepsilon_k \biggr] .
\end{equation}
Define the mean width of $\mathcal{D}$ relative to $\boldsymbol{\varepsilon}$ to be the 
limiting linear growth rate of the finite sample mean widths,
\begin{equation} \label{MW}
\kappa(\mathcal{D} : \boldsymbol{\varepsilon}) = \lim_n \frac{1}{n} \kappa_n(\mathcal{D} : \boldsymbol{\varepsilon}),
\end{equation}
which exists by subadditivity (see Remark \ref{Rmk:Love}).
When $\varepsilon_i \sim N(0,1)$ we write $\kappa(\mathcal{D} : \boldsymbol{\varepsilon})$ as 
$\kappa_{G}(\mathcal{D})$ and refer to this quantity as the Gaussian 
mean width of the family $\mathcal{D}$.  
\end{defn}

\vskip.1in

Finite sample mean widths have been widely studied in machine learning and empirical process theory, with an emphasis on Rademacher and 
Gaussian noise \cite{boucheron2013concentration,ledoux2013probability}.  The mean width of $\mathcal{D}$ has close
connections with the entropy of $\mathcal{D}$.

\begin{thm} \label{Thm:Furstenberg}
Let  $\boldsymbol{\varepsilon} = (\epsilon_k)_{k \geq 0}$ be i.i.d.\ with mean zero and finite variance.
If $h(\mathcal{D}) = 0$ then $\kappa(\mathcal{D} : \boldsymbol{\varepsilon}) = 0$.   
Moreover, the Gaussian mean width $\kappa_{G}(\mathcal{D}) = 0$ if and only if
$h(\mathcal{D}) = 0$.
\end{thm}

\begin{rmk}
Theorem \ref{Thm:Furstenberg} establishes a type of qualitative relationship between asymptotic mean width and entropy: for a given family of dynamical models, they are either both zero or both positive.
In general one cannot expect a more quantitative relationship between asymptotic mean width and entropy. 
While it is possible to provide upper and lower bounds on $\kappa_n(\mathcal{D} : \boldsymbol{\varepsilon})$ in terms of $\ell_2$ covering numbers (as in the proof of Theorem \ref{Thm:Furstenberg}), additional care must be taken when passing to the limits to obtain the mean width and the entropy. As it turns out, the presence of these limits in the definitions precludes any more quantitative dependence between these quantities.

One way to see that no quantitative dependence is possible between these quantities is to observe that the entropy $h(\mathcal{D})$ is invariant under any continuous, invertible 
change of coordinates 
(see \cite[p. 167]{Walters1982}), whereas the asymptotic mean width
$\kappa_{G}(\mathcal{D})$ is not invariant under such operations.  
For example, the asymptotic mean width scales linearly  
if the observation functions $f_{\theta}$ are multiplied by a fixed constant, but the 
entropy remains unchanged.
\end{rmk}

Let $\gamma_2$ denote the $\ell$-distortion in the special case that $\ell$ is the squared loss. Note that 
$\gamma_2(\mathbf{U},\mathbf{V})^{1/2}$ is in fact a metric on the space of $\R$-valued stationary 
stochastic processes.   As the squared loss is a Bregman divergence, minimum risk fitting of a zero
entropy family will converge to the optimal parameter set for the signal by Theorem \ref{Thm:SPEC}.
The next theorem extends this result to the case where the mean width of the family is zero.

\begin{thm} 
\label{Thm:LSConsistency}
Let $\mathbf{Y} = \mathbf{V}+ \boldsymbol{\varepsilon}$, where 
$\mathbf{V}$ is ergodic and $\boldsymbol{\varepsilon}$ is an i.i.d. process with mean zero and finite variance. 
If $\kappa(\mathcal{D} : \boldsymbol{\varepsilon}) = 0$, then any sequence 
of least squares estimators converges almost surely to 
$\Theta_2(\mathbf{V})$. 
\end{thm}

In our final result of this section, we establish the consistency of least squares estimation for a family of transformations on a compact state space in $\R^d$ where each observation function is the identity. 
Suppose $\cX \subset \R^d$ is compact and $\{ T_{\theta} : \theta \in \Theta\}$ is a family 
of transformations on $\cX$ such that $\Theta$ is a compact metric space and $(\theta,x) \mapsto T_{\theta}(x)$ is continuous.
Further, suppose that $Y_k = T_{\theta^*}^k(X) + \varepsilon_k$ where
$X$ is distributed according to an ergodic measure $\mu \in \mathcal{M}(\mathcal{X},T_{\theta^*})$, and 
$\boldsymbol{\varepsilon} = (\varepsilon_k)_{k \geq 0}$ is i.i.d. with mean zero and finite variance 
and is independent of $X$.

\begin{cor} \label{Cor:Rd}
If the topological entropy of $T_{\theta}$ is zero for all $\theta \in \Theta$, then any sequence 
of least squares estimators converges almost surely to the set
$
\{ \theta \in \Theta : \mu(T_{\theta} = T_{\theta^*}) = 1 \}
$.
\end{cor}

The limit set in Corollary \ref{Cor:Rd} contains $\theta^*$ and serves as the natural identifiability class of $\theta^*$ in this setting.

\section{Examples of Dynamical Models} 
\label{Sect:PositiveExamples}

In order to complement the main results of this paper, which are of a general nature, we present here several specific examples of families of dynamical models satisfying (D1)-(D3) 
that capture regularities of interest. Under some additional assumptions, these families have entropy $h(\mathcal{D})=0$, 
and therefore all of our main results apply.
In a number of cases, these or similar families have been fit to data by applied scientists (e.g. \cite{Brackley2010,Letham2016,McGoff_GRNs,Turchin2013}), 
albeit without any theoretical guarantees of consistency.

\begin{example}(Toral rotations and almost periodicity)
Let the state space $\mathcal{X}$ be the $d$-dimensional torus $\mathbb{T}^d$, which is the direct product of $d$ circles, $\mathbb{T}^d = S^1 \times \dots \times S^1$. For a vector $\alpha \in \mathbb{T}^d$, define the transformation $R_{\alpha} : \mathbb{T}^d \to \mathbb{T}^d$ to be the rotation of $\mathbb{T}^d$ by the angle vector $\alpha$, i.e. $R_{\alpha}(x) = x+\alpha$ (addition in $\mathbb{T}^d$). Then let $\mathcal{F} \subset C(\mathbb{T}^d)$ be a compact set of continuous functions from $\mathbb{T}^d$ to $\R$ (with respect to the topology induced by the supremum norm). Let $\Theta = \mathbb{T}^d \times \mathcal{F}$, and define the family of dynamical models $\mathcal{D} = \{ (R_{\alpha},f) : (\alpha,f) \in \Theta \}$. With these definitions, $\mathcal{D}$ is a continuous family of dynamical models, and $h(\mathcal{D}) = 0$. Fitting this family to a process amounts to looking for periodic or ``almost periodic" (also known as ``quasi-periodic") structure in the observations. Intuitively, one is looking for up to $d$ independent ``periods" in a process. An observed process would have $d$ independent ``periods" if there were $d$ periodic processes with incommensurate periods and the observed process is a function of all $d$ of these periodic processes.

As a specific example of a setting in which such models might arise, one may consider restricted classes of dynamic gene regulatory networks that exhibit periodic behavior. 
Inference of gene regulatory networks from observed data is considered an important problem in systems biology \cite{Marbach2012}. In recent years, it has become increasingly feasible for experimentalists to assay the abundance of all the genes in a given system with regular frequency over time. In such cases, one would like to infer the structure of the underlying network from the observed gene expression dynamics \cite{McGoff_GRNs}.  

In many situations, one would expect gene regulatory networks to be zero entropy systems. In particular, the networks studied in chronobiology should exhibit periodic dynamics by definition. Examples of such systems include the cell cycle and circadian oscillators.
\end{example}

\begin{example}(Subcritical logistic family and ecology)
Since at least the early work of May \cite{May1974}, simple parametric families of dynamical systems have been used by ecologists as models of the population dynamics of many species \cite{Levin2009}. 
In many instances, various types of deterministic models have been fit to ecological data (e.g. \cite{Turchin2013}).

The prototypical family in this context is the logistic family, which may be parametrized as follows. Consider the state space $\mathcal{X} = [0,1]$ and the family of maps $T_{a} : [0,1] \to [0,1]$, where $T_a(x) = a x (1-x)$ for $a \in [0,4]$. If we restrict $a$ to the region $[0,3.5]$, then the family of dynamical models will have zero entropy. This situation is thought to occur in many naturally occurring populations (see results and discussion from \cite{Hassell1976}). In examples such as these, the state variable $x$ typically represents the (rescaled) population size. The overall structure of the logistic family captures the idea that the reproductive rate depends on the density of the population, taking into account effects such as competition for limited resources. Given observations of population size over time, one may try to fit these dynamical models to the observations to identify the parameter $a$.
\end{example}

\begin{example}(Symbolic dynamics and quasicrystals)
Symbolic dynamical systems, also known as subshifts, are a useful family of models that
arise in the study of dynamical systems through discretizing of the state space. 
Informally, if $ T : X \to X$ is a dynamical system and $\{A_1,\dots,A_N\}$ is a finite partition of $X$, then the associated symbolic system 
consists of the label sequences $\{ (\pi(T^k x) )_k : x \in X \}$ under the left shift map, where 
$\pi: X \to \{1,\dots, N\}$ is defined by the relation $x \in A_{\pi(x)}$.  
Symbolic systems have been widely studied for their own sake \cite{Lind1995}, for the purpose of understanding other dynamical systems \cite{Bowen1975}, and for their connections to other disciplines, e.g., physics \cite{Ruelle2004}. 
Due to their combinatorial nature, they can be used to model a variety of regularities 
in physical systems. For example, they have been used in communications, coding and information theory to capture the rules by which binary strings should be encoded on magnetic tapes and compact discs in order to minimize errors \cite{Lind1995}.

As another recent example, symbolic dynamical systems have recently been used by several researchers \cite{Damanik2015,Robinson1996,Solomyak1998} 
as a mathematical model of quasicrystals, which were discovered by 
Shechtman \cite{Shechtman1984}.
Quasicrystals are characterized by the presence of long-range aperiodic order, in contrast to crystals, which are characterized by 
long-range periodic order.  Long-range aperiodic order is found in substitution systems \cite{Queffelec2010}, 
which are constructed by enforcing a rigid hierarchical structure at all scales. Systems exhibiting this type of long-range order typically have zero entropy.
\end{example}

\section{Discussion of rates and related work} 
\label{Sect:RelatedWork}




The results of this paper have points of overlap with recent work in the statistics and
machine learning literature concerning estimation, forecasting, and prediction from
dependent observations.  While some of this work, for example
Morvai and Weiss \cite{Morvai2005_2, Morvai2005_3}, 
Nobel \cite{Nobel2006}, and Adams and Nobel \cite{Adams2010}, 
is focused on asymptotics for general ergodic observations, a number of papers provide rates of convergence or
finite sample bounds under more stringent assumptions.

Modha and Masry \cite{ModMas98}, Meir \cite{Meir00}, and Alquier and Wintenberger \cite{Alquier2012} 
establish oracle inequalities and finite sample bounds for predicting the next value of a stationary process.
Agarwal and Duchi \cite{Agarwal2013}, 
Kuznetsov and Mohri \cite{Kuznetsov2015,Kuznetsov2017,Kuznetsov2016}, 
and Zimin and Lampert \cite{Zimin2017}
establish finite sample performance bounds on the conditional risk of online learning algorithms
for predicting dependent time series.  Each of the papers cited above imposes mixing conditions
on the observations as well as regularity conditions on the loss function and model family of interest. 
Shalizi and Kontorovich \cite{Shalizi2013} consider learning mixtures of stationary processes, while
Kontorovich \cite{Kontorovich2012} studies statistical
estimation using finite automata with bounded memory.
Hang and Steinwart \cite{Hang14} obtain rates of convergence for empirical risk minimization
from $\alpha$-mixing observations, while Wong {\it et al.}\ \cite{Wong2016} establish finite sample bounds for
Lasso-based inference under $\beta$-mixing conditions.
In another direction, Rakhlin {\em et al.}\ \cite{RST14} and Rakhlin and Sridharan \cite{Rakhlin2017} have 
established exponential inequalities for suprema of martingale difference sequences by using and extending
ideas from machine learning, including Rademacher complexity and deterministic regret inequalities.

As noted in the introduction, the problem of fitting dynamical models
differs from the inference problems above as both the observations {\it and} the models under study
can exhibit dynamical behavior and long-range dependence. 
Moreover, our principal results make
no assumptions concerning mixing properties of the observed process $\mathbf{Y}$, 
mixing properties or stationary distributions of the dynamical models $\mathcal{D}$, smoothness 
of the loss $\ell(\cdot,\cdot)$ (beyond lower semicontinuity), or the relationship between the observed
process and the family of models being fit.
This general setting enables us to study the asymptotic behavior of minimum risk estimation 
for dynamical models in a variational framework where the roles of the
observed process, the loss, and (most critically) the model family are clear and easy to understand.

The results here provide a framework for, and initial progress towards, the
detailed analyses of specific problems and model families
that might lead to rates of convergence, or finite sample performance bounds.
It is evident from the papers above that stronger results, e.g., rates of convergence,
will require substantially stronger assumptions, including mixing conditions (with geometric or
polynomial rates) on the observed process, smoothness (possibly with convexity) of the loss 
function, and stronger, covering-based complexity constraints on the family of dynamical models.  
If mixing type conditions are required for the dynamical models themselves, but these would require
additional assumptions, as mixing conditions typically hold only for distinguished invariant 
measures or observation functions.

A number of the papers cited above make use of exponential probability bounds, typically Azuma-Hoeffding type 
inequalities, to control error terms that are sums of martingale differences.  Martingale differences
to not arise in the theoretical analysis of the paper, but we note that there are some uses of reverse martingale 
methods in the dynamics literature \cite{Liverani1996}.
Investigating martingale approaches to the problems considered here represents an interesting direction 
for future research.

Our work is also related to a line of research concerning least squares estimation of individual 
sequences from noisy observations, see for example \cite{Pollard2006,VanDeGeer1990,Wu1981}.  
Pollard and Radchenko \cite{Pollard2006} use empirical process theory to establish 
consistency and asymptotic normality of least squares estimation for individual sequences from signal plus noise. 
In the present work, we consider sets of individual sequences that arise from a continuous family of dynamical models, as in (\ref{Eqn:U}), and we are interested in inference of a dynamical invariant parameter (i.e. $\theta$), rather than the signal sequence itself. 

Furstenberg's original work on joinings \cite{Furstenberg1967} includes an application of joinings to a nonlinear filtering problem. 
Beyond this application, we are not aware of other uses of joinings in the literature on statistical inference.
Ornstein and Weiss \cite{OrnsteinWeiss1990} studied the estimation of a stochastic process from its samples. 
They proposed an inference procedure, based on matching $k$-block frequencies, and characterize when it produces 
consistent estimates of the observed stochastic process in the $d$-bar metric. 

Some of Furstenberg's original results are extended in recent work of Lev, Peled, and Peres \cite{Lev2015}. 
Given an infinite sequence equal to a target signal plus noise, they consider the problem of detecting whether the signal is non-zero, 
and the problem of recovering the signal from the given sequence.  Target
sequences are assumed to belong to a known family (as in \cite{Pollard2006}), and
their analysis places no restrictions (beyond measurability) on the 
detection and filtering procedures, which can be functions of the entire sequence of observations.  

Finally, we mention that statistical inference in the context of dynamical systems has been considered in a variety of subject areas; see the survey \cite{McGoffSurvey2015} for a broad overview and references.  Dynamical systems in the observational noise setting have been studied in \cite{Lalley1999,LalleyNobel2006,McGoff2015}, and statistical prediction in the context of dynamical systems has been considered in \cite{Hang2016,Hang2016_2,Steinwart2009,Viswanath2013}. 

\subsection{Generalizations and future work}

Generalization of all of the definitions and results of the paper to $\R^d$-valued models and processes is straightforward, requiring only minor changes of notation.  We omit the details.
In a different direction, one could analyze families of dynamical models defined on a 
non-compact state space $\cX$ with uniformly bounded observation functions, 
requiring only measurability of the maps
$(\theta,x) \mapsto T_{\theta}(x)$ and $(\theta,x) \mapsto f_{\theta}(x)$. 
For families $\mathcal{D}$ of this more general type, the set $\mathcal{U}_{\mathcal{D}}$ 
of associated sequences would not necessarily be a closed (hence compact) subset of $\R^{\N}$, 
and in this case one needs to consider the closure of $\mathcal{U}_{\mathcal{D}}$, 
along with all the stationary processes supported on this set.
The analysis here can be carried out in this more general setting, but the corresponding 
results are difficult to interpret in the context of the original inference problem.

%
%
%

\section{Optimal tracking and proof of Theorem \ref{Thm:GEN}} \label{Sect:Connections}

In this section we discuss the connections between fitting dynamical models and the optimal tracking problem studied in \cite{McGoffNobel}. In particular, we construct a single dynamical system that captures the important features of the family $\mathcal{D}$ of dynamical models, and we show how this system may be analyzed in the context of optimal tracking.

\subsection{Optimal tracking}

The tracking problem for dynamical systems concerns two systems: a model system $T : \mathcal{Z} \to \mathcal{Z}$, where $\mathcal{Z}$ is compact and metrizable and $T$ is continuous, and an observed system $S : \mathcal{Y} \to \mathcal{Y}$, where $\mathcal{Y}$ is a separable completely metrizable space and $S$ is Borel measurable. Given an initial segment of a trajectory $y, S(y), \dots, S^{n-1}(y)$ from the observed system, one seeks a corresponding initial condition $z_n$ such that the trajectory $z_n, T(z_n), \dots, T^{n-1}(z_n)$ from the model system ``tracks" the given trajectory from the observed system. An optimal tracking trajectory is chosen by minimizing an additive cost functional 
\begin{equation*}
\sum_{k = 0}^{n-1} c\bigl( S^ky, T^kz \bigr),
\end{equation*}
where $c : \mathcal{Y} \times \mathcal{Z} \to \R$ is a fixed lower semicontinuous cost function.

The results from \cite{McGoffNobel} consider the situation when the observed initial condition $y$ is drawn from an ergodic measure $\nu \in \mathcal{M}(\mathcal{Y},S)$ and $\sup_z |c(y,z)|$ is bounded above by a function in $L^1(\nu)$.
Here we state a version of the previous results that is sufficient for our purposes, for which we require a bit more notation. Let $\Theta$ be a compact metrizable space, and let $\varphi : \mathcal{Z} \to \Theta$ be a continuous map satisfying $\varphi \circ T = \varphi$. We denote by $\mathcal{J}(\nu : \theta)$ the set of joinings of the process $\{S^k(Y_0)\}_{k \geq 0}$, where $Y_0 \sim \nu$, with any process of the form $\{T^k(Z_0)\}_{k\geq 0}$, where $Z_0 \sim \mu$ for some $\mu \in \mathcal{M}(\mathcal{Z},T)$ such that $\mu(\varphi^{-1}\{\theta\}) = 1$.

\begin{CorThm}[\cite{McGoffNobel}]
\label{Thm:Tracking}
Let  $T: \mathcal{Z} \to \mathcal{Z}$, $S: \mathcal{Y} \to \mathcal{Y}$, $c : \mathcal{Y} \times \mathcal{Z} \to \R$, $\nu$, $\Theta$, and $\varphi : \mathcal{Z} \to \Theta$ be as above. 
If $\hat{z}_n = \hat{z}_n(y,\dots,S^{n-1}y)$ and the following equality holds $\nu$ almost surely,
\begin{equation} \label{Eqn:Rage}
\lim_n  \frac{1}{n} \sum_{k = 0}^{n-1} c(S^k y, T^k \hat{z}_n) = \lim_n \inf_{z \in \mathcal{Z}} \frac{1}{n} \sum_{k = 0}^{n-1} c(S^ky, T^kz),
\end{equation}
then $\hat{\theta}_n = \varphi( \hat{z}_n)$ converges ($\nu$ almost surely) to the non-empty, compact set
\begin{equation}
\Theta_{min} = \argmin_{\theta \in \Theta} \min_{\mathcal{J}(\nu : \theta)} \mathbb{E} \bigl[ c(Y_0,Z_0) \bigr].
\end{equation}
Furthermore, for any $\theta \in \Theta_{min}$, there exists $\hat{z}_n$ such that (\ref{Eqn:Rage}) holds and $\hat{\theta}_n = \varphi( \hat{z}_n)$ converges to $\theta$.
\end{CorThm}

\subsection{Proof of Theorem \ref{Thm:GEN} }

To be begin the proof, we describe how fitting a continuous family of dynamical models to an observed stochastic process can be cast as a tracking problem. 
As an important first step, we define a single dynamical system that encapsulates the entire family of dynamical models. Consider the state space
\begin{equation*}
\mathcal{Z} = \biggl\{ \Bigl( \theta , \, \bigl( f_{\theta} \circ T_{\theta}^k(x) \bigr)_{k \geq 0} \Bigr) : \theta \in \Theta, \, x \in \cX \biggr\} \subseteq \Theta \times \R^{\N},
\end{equation*}
and define the transformation $T : \mathcal{Z} \to \mathcal{Z}$ by 
$T(\theta, (u_k)_{k \geq 0}) = (\theta, (u_{k+1})_{k \geq 0})$, which is clearly continuous (in the product topology). 
We now establish some basic properties of the dynamical system $(\mathcal{Z},T)$ associated 
with the family $\mathcal{D}$ of dynamical models.

\begin{lemma} \label{Lemma:BigShift}
The set $\mathcal{Z}$ is a compact subset of $\Theta \times \R^{\N}$ when $\R^{\N}$ is equipped with the product topology, 
and the map $T$ is continuous. 
If $\mu$ is an ergodic element of $\mathcal{M}(\mathcal{Z},T)$, then there exists $\theta \in \Theta$ and an ergodic process $\mathbf{U} \in \mathcal{Q}_{\theta}$ with distribution $\nu$ such that $\mu = \delta_{\theta} \otimes \nu$. 
\end{lemma}

\begin{proof}
 By our hypotheses on the family $\mathcal{D}$, both the parameter space $\Theta$ and the state space $\cX$ are compact, and therefore $\Theta \times \cX$ is compact. Define the map $\pi : \Theta \times \cX \to \Theta \times \R^{\N}$ by 
 \begin{equation*}
  \pi(\theta,x) = \Bigl( \theta, \bigl( f_{\theta} \circ T_{\theta}^k(x) \bigr)_{k \geq 0} \Bigr).
 \end{equation*}
 It is clear from the definition of $\mathcal{Z}$ that $\mathcal{Z}$ is the image of $\Theta \times \cX$ under $\pi$. To show that $\mathcal{Z}$ is compact, it now suffices to check that $\pi$ is continuous.
 
 Let $\{(\theta_n,x_n)\}_{n \geq 1}$ be a sequence converging to $(\theta,x)$ in $\Theta \times \cX$. Let $K \in \N$. The continuity conditions (D2) and (D3) imply that for $0 \leq k \leq K$, 
 \begin{equation*}
  \lim_n f_{\theta_n} \circ T_{\theta_n}^k(x_n) = f_{\theta} \circ T_{\theta}^k(x).
 \end{equation*}
 As $K$ was arbitrary, we have shown that $\{\pi(\theta_n,x_n)\}_{n \geq 1}$ converges to $\pi(\theta,x)$ in $\Theta \times \R^{\N}$ in the product topology, and therefore $\pi$ is continuous.
 
 The left-shift $\tau : \R^{\N} \to \R^{\N}$ is continuous, and therefore $T = (\Id \times \tau)|_{\mathcal{Z}}$ is continuous.
 
 For the last statement of the lemma, define the map $R : \Theta \times \cX \to \Theta \times \cX$ as the skew-product over the identity, $R(\theta,x) = (\theta, T_{\theta}(x))$. By construction, we have $\pi \circ R = T \circ \pi$. Thus, $\pi$ is a factor map from $(\Theta \times \cX,T)$ onto $(\mathcal{Z},S)$. It follows that the push-forward map from $\mathcal{M}(\Theta \times \cX,R)$ to $\mathcal{M}(\mathcal{Z},T)$ (given by $\eta \mapsto \eta \circ \pi^{-1}$) is a surjection \cite[p. 19]{DGS}. 
 
 Now let $\mu \in \mathcal{M}(\mathcal{Z},T)$ be ergodic. Since the push-forward map from $\mathcal{M}(\Theta \times \cX,R)$ to $\mathcal{M}(\mathcal{Z},T)$ (given by $\eta \mapsto \eta \circ \pi^{-1}$) is a surjection, there exists an ergodic $\eta \in \mathcal{M}(\Theta \times \cX,R)$ such that $\eta \circ \pi^{-1} = \mu$. 
 Since $\proj_{\Theta} \circ T = \proj_{\Theta}$, the induced measure $\eta \circ  (\proj_{\Theta} \circ \pi)^{-1}$ on $\Theta$ must be invariant under the identity map. Also, it must be ergodic, since $\eta$ is ergodic. As the only ergodic measures for the identity map are the point masses, we see that there exists $\theta \in \Theta$ such that $\eta \circ (\proj_{\Theta} \circ \pi)^{-1} = \delta_{\theta}$. Then $\eta = \delta_{\theta} \otimes \xi$ for some ergodic measure $\xi \in \mathcal{M}(\mathcal{X},T_{\theta})$.  Finally, we conclude that $\mu = \eta \circ \pi^{-1} = \delta_{\theta} \otimes \nu$, where $\nu$ is the distribution of an ergodic process in $\mathcal{Q}_{\theta}$.
\end{proof}

We now proceed with the proof of Theorem \ref{Thm:GEN}.
The observed process $\mathbf{Y}$ gives rise to an observed dynamical system in the tracking problem, where $\mathcal{Y} = \R^{\N}$, $S : \mathcal{Y} \to \mathcal{Y}$ is the left shift $S( (u_k)_{k \geq 0}) = (u_{k +1})_{k \geq 0}$, and $\nu$ is the distribution of $\mathbf{Y}$ on $\R^{\N}$. Finally, as a cost function, we choose $c : \mathcal{Y} \times \mathcal{Z} \to \R$, where
\begin{equation*}
c\Bigl( \mathbf{v}, \bigl( \theta, \mathbf{u} \bigr) \Bigr) = \ell(u_0,v_0)
\end{equation*}
Then an application of Theorem \ref{Thm:Tracking} yields that
any sequence of minimum $\ell$-risk parameters $(\hat{\theta}_n)_{n \geq 1}$ converges almost surely to the set $\Theta_{\ell}(\mathbf{Y})$. Furthermore, this projection is nonempty and compact, and Theorem \ref{Thm:GEN} (2) holds. We have thus proved Theorem \ref{Thm:GEN}.

\section{Entropy and mean width (proofs)} \label{Sect:EntropyProofs}

In this section we study the notions of entropy and mean width for families of dynamical models. We begin with entropy.

\subsection{Entropy for families of dynamical models}

Recall that the pseudo-metrics $d_{n,p}(\cdot,\cdot)$ and the $\ell_p$ entropies $h_p(\mathcal{D})$ were defined in Section \ref{Sect:EntropyDef}. Additionally, we define $B_{n,p}(\mathbf{u},r)$ to be the $d_{n,p}(\cdot, \cdot)$-ball of radius $r$ centered at $\mathbf{u} \in \R^{\N}$.

Before proving the main results presented in Section \ref{Sect:EntropyDef}, we make a simple observation.
For $p \geq 1$, we have that $d_{n,p}(\mathbf{u},\mathbf{v}) \leq d_{n,\infty}(\mathbf{u},\mathbf{v})$. 
Hence, for any $\mathcal{U} \subset \R^{\N}$ and $\delta >0$, 
\begin{equation} \label{Eqn:RedTailed}
 N(\mathcal{U},\delta,d_{n,p}) \leq N(\mathcal{U},\delta,d_{n,\infty}).
\end{equation}

The following lemma, which bounds the cardinality of $d_{n,\infty}(\cdot,\cdot)$-separated sets that are contained in a single $d_{n,p}(\cdot,\cdot)$-ball, is used to prove Theorem \ref{Thm:EntropyEquality}. For notation, let $M^{\infty}_n(\mathcal{U},\delta)$ denote the maximum cardinality of any set in $\mathcal{U}$ that is $\delta$-separated with respect to $d_{n,\infty}(\cdot,\cdot)$.

\begin{lemma} \label{Lemma:L2Linfty}
Let $p \in [1,\infty)$, $K \geq 1$ and $\delta \in (0,1)$. Set $\epsilon = (\delta/2)^{(1+p)/p}$. Then for any $\mathbf{u} \in [-K,K]^{\N}$ and $n \geq 1$,
\begin{equation*}
 M_{n}^{\infty}(B_{n,p}(\mathbf{u},\epsilon),\delta) \leq  (3K/\delta)^{\delta n/2} \cdot 2^{H(\delta/2)n},
\end{equation*}
where $H(x) = -x \log x - (1-x) \log (1-x)$.
\end{lemma}

\begin{proof}
Suppose $\mathbf{u} \in [-K,K]^{\N}$ and there exists $\{\mathbf{v}_1,\dots,\mathbf{v}_M\} \subset B_{n,p}(\mathbf{u},\epsilon)$ such that $d_{n,\infty}(\mathbf{v}_i,\mathbf{v}_j) \geq \delta$ for $i \neq j$.
For the sake of notation, let $\mathbf{u} = (u(k))_{k \geq 0}$ and $\mathbf{v}_j = (v_j(k))_{k \geq 0}$. By subtracting $\mathbf{u}$ from all these sequences if necessary, we assume without loss of generality that $u(k) = 0$ for all $k$. 

The idea of the proof is to bound $M$ by estimating the number of coordinates of each $\mathbf{v}_j$ that deviate from $0$ by more than $\delta/2$. To begin, we define the following subsets of $[-K,K]$: $A_0 = [-\delta/2,\delta/2)$, and $A_r = [-K+(r-1)\delta,-K+r\delta) \setminus A_0$, for $r = 1, \dots, s$, where $s= \lceil 2K/\delta \rceil$. Now we code the points $\mathbf{v}_j$ according to this partition: define $r_j(k)$ by the relation $v_j(k) \in A_{r_j(k)}$, and let $\pi : \{\mathbf{v}_1,\dots,\mathbf{v}_M\} \to \{0,\dots,s\}^n$ be given by
\begin{equation*}
 \pi(\mathbf{v}_j) = (r_j(0),\dots,r_j(n-1)).
\end{equation*}

First, observe that $\pi$ is injective. Indeed, if $i \neq j$, then $d_n^{\infty}(\mathbf{v}_i,\mathbf{v}_j) \geq \delta$, and hence there exists $k$ such that $|v_i(k) - v_j(k)| \geq \delta$, which implies $r_i(k) \neq r_j(k)$.
Second, observe that
\begin{equation} \label{Eqn:Golden}
 s = \lceil 2K/\delta \rceil \leq 2K/\delta + 1 \leq 3K/\delta.
\end{equation}

Now we proceed with the main bounds. Since $\{\mathbf{v}_1,\dots,\mathbf{v}_M\} \subset B_{n,p}(\mathbf{u},\epsilon)$, we have that for each $j$,
\begin{equation*}
 \epsilon^p n \geq \sum_{k=1}^n |v_j(k)|^p.
\end{equation*}
Furthermore, by construction, if $r_j(k) \neq 0$, then $|v_j(k)| \geq \delta/2$, and therefore
\begin{equation*}
  \epsilon^p n \geq \sum_{k=1}^n |v_j(k)|^p \geq (\delta/2)^p \bigl| \{ k : r_j(k) \neq 0 \} \bigr|.
\end{equation*}
From this inequality and the choice of $\epsilon$, we deduce that 
\begin{equation} \label{Eqn:Perigrine}
\{\pi(\mathbf{v}_1),\dots,\pi(\mathbf{v}_M)\} \subset \biggl\{ z \in \{0,\dots,s\}^n :  \bigl| \{ k : z(k) \neq 0\} \bigr| \leq \delta n /2 \biggr\}.
\end{equation}
By the facts established above (injectivity of $\pi$, (\ref{Eqn:Perigrine}), and (\ref{Eqn:Golden})) and a well-known inequality for binomial sums (see \cite{Shields1996}),
\begin{align*}
 M & = \bigl| \{\pi(\mathbf{v}_1),\dots,\pi(\mathbf{v}_M)\} \bigr| \leq \sum_{k=0}^{\delta n /2} \binom{n}{k} s^k \leq s^{\delta n/2}  \sum_{k=0}^{\delta n /2} \binom{n}{k} \\
 & \leq (3K/\delta)^{\delta n /2}  \sum_{k=0}^{\delta n /2} \binom{n}{k} \leq   (3K/\delta)^{\delta n /2} \cdot 2^{H(\delta/2) n},
\end{align*}
which completes the proof of the lemma.
\end{proof}

With this lemma in place, we now turn to the proof of Theorem \ref{Thm:EntropyEquality}.

\vspace{2mm}

\begin{PfofEntropyEquality}

The inequality $h_p(\mathcal{D}) \leq h_{\infty}(\mathcal{D})$ follows easily from (\ref{Eqn:RedTailed}) and the definition of entropy. The remainder of the proof is devoted to showing the reverse inequality. As in the definition of entropy, we let
\begin{equation*}
\mathcal{U} = \Bigl\{ \bigl(f_{\theta} \circ T_{\theta}^k(x) \bigr)_{k \geq 0} : x \in \cX, \, \theta \in \Theta \Bigr\}.
\end{equation*}
Since $\cX$ and $\Theta$ are compact and the map $(\theta,x) \to f_{\theta}(x)$ is continuous, there exists $K\geq 1$ such that $|f_{\theta}(x)| \leq K$ for all $x \in \cX$ and $\theta \in \Theta$. Thus $\mathcal{U} \subset [-K,K]^{\N}$.

Let $\delta \in (0,1)$, and let $\{\mathbf{v}_1, \dots, \mathbf{v}_M\}$ be a maximal $\delta$-separated set for $\mathcal{U}$ with respect to $d_{n,\infty}( \cdot, \cdot)$. Note that $M = M_{n}^{\infty}(\mathcal{U},\delta)$. 
Now let $\epsilon = (\delta/2)^{(1+p)/p}$, and let $\{\mathbf{u}_1,\dots,\mathbf{u}_L\}$ be an $\epsilon$-covering set for $\mathcal{U}$ with respect to $d_{n,p}(\cdot, \cdot)$ with minimal cardinality. Note that $L = N(\mathcal{U},\epsilon,d_{n,p})$.
By the union bound, 
\begin{align}
\begin{split} \label{Eqn:Goshawk}
M_{n}^{\infty}(\mathcal{U},\delta) & = \Biggl| \bigcup_{i=1}^L B_{n,p}(\mathbf{u}_i,\epsilon) \cap \{ \mathbf{v}_1, \dots, \mathbf{v}_M\} \Biggr| \\
& \leq \sum_{i = 1}^L \Bigl| B_{n,p}(\mathbf{u}_i,\epsilon) \cap \{ \mathbf{v}_1, \dots, \mathbf{v}_M\} \Bigr| \\
& \leq L \cdot \max \biggl\{ \Bigl| B_{n,p}(\mathbf{u}_i,\epsilon) \cap \{ \mathbf{v}_1, \dots, \mathbf{v}_M\} \Bigr| : i \in \{1,\dots,L\} \biggr\} .
\end{split}
\end{align}
Applying Lemma \ref{Lemma:L2Linfty} and the fact that $L = N(\mathcal{U},\epsilon,d_{n,p})$ in (\ref{Eqn:Goshawk}), we obtain
\begin{align} \label{Eqn:Harpy}
M_{n}^{\infty}(\mathcal{U},\delta) \leq N(\mathcal{U},\epsilon,d_{n,p}) (3K/\delta)^{\delta n/2} 2^{H(\delta/2) n}.
\end{align}
Since any maximal $\delta$-separated set must be a $\delta$-covering set, we have $N(\mathcal{U},\delta,d_{n,\infty}) \leq M_{n}^{\infty}(\mathcal{U},\delta)$, and then from (\ref{Eqn:Harpy}), we see that
\begin{equation*}
N(\mathcal{U},\delta,d_{n,\infty}) \leq N(\mathcal{U},\epsilon,d_{n,p}) (3K/\delta)^{\delta n/2} 2^{H(\delta/2) n}.
\end{equation*}
Taking logarithm and dividing by $n$ yields
\begin{equation*}
\frac{1}{n} \log N(\mathcal{U},\delta,d_{n,\infty}) \leq \frac{1}{n} \log N(\mathcal{U},\epsilon,d_{n,p}) + \frac{\delta}{2} \log (3K / \delta) + H(\delta/2) \log 2.
\end{equation*}
Thus letting $n$ tend to infinity gives
\begin{equation*}
h_{\infty}(\mathcal{U},\delta) \leq h_p(\mathcal{U},\epsilon) +  \frac{\delta}{2} \log (3K / \delta) + H(\delta/2) \log 2.
\end{equation*}
Since $\epsilon = (\delta/2)^{(1+p)/p}$, taking the limit as $\delta$ decreases to zero allows us to conclude that $h_{\infty}(\mathcal{D}) \leq h_p(\mathcal{D})$.
\end{PfofEntropyEquality}

\subsection{Variational characterization of mean width}

Here we collect a few facts regarding mean width, which are used elsewhere. Let us begin with the fact that the sequence of finite sample mean widths is subadditive.
\begin{rmk} \label{Rmk:Love}
Using definition (\ref{kappan-def}), one may easily check that that for $m,n \geq 1$,
\begin{align*}
\kappa_{m+n}(\mathcal{D} : \boldsymbol{\varepsilon}) & \leq 
\mathbb{E} \Biggl[ \sup_{x,\theta} \sum_{k=0}^{m-1}  f_{\theta} \circ T_{\theta}^k(x) \cdot \varepsilon_k \Biggr]  + \mathbb{E} \Biggl[ \sup_{x,\theta} \sum_{k=m}^{m+n-1} f_{\theta} \circ T_{\theta}^k(x) \cdot \varepsilon_k  \Biggr]
 \\
& \leq  
\kappa_{m}(\mathcal{D} : \boldsymbol{\varepsilon}) + 
\kappa_{n}(\mathcal{D} : \boldsymbol{\varepsilon}) .
\end{align*}
The last inequality above is a consequence of the stationarity of $\boldsymbol{\varepsilon}$ and the 
fact that $f_{\theta} \circ T_{\theta}^{k + m}(x) = f_{\theta} \circ T_{\theta}^k ( T_{\theta}^m (x))$ 
for any $x \in \cX$ and $\theta \in \Theta$.  Thus the sequence 
$\{ \kappa_{n}(\mathcal{D} : \boldsymbol{\varepsilon}) : n \geq 1 \}$ is subadditive, and therefore
the limit in (\ref{MW}) exists.
\end{rmk}


The following result provides a variational characterization of the mean width.

\begin{thm} \label{Thm:KappaVar}
If $\mathcal{D}$ is a family of dynamical models satisfying (D1)-(D3) and 
$\boldsymbol{\varepsilon}$ is a stationary ergodic process with finite mean, then
\begin{equation*}
\kappa(\mathcal{D} : \boldsymbol{\varepsilon}) 
\ = 
\sup_{\mathbf{U} \in \mathcal{Q}_{\mathcal{D}}} 
\sup_{\cJ(\mathbf{U}, \, \boldsymbol{\varepsilon})} 
\mathbb{E} \bigl[ U_0 \cdot \varepsilon_0 \bigr],
\end{equation*}
and the supremum is achieved.
\end{thm}

\begin{proof}
Using the same system $(\mathcal{Z},T)$ appearing in Section \ref{Sect:Connections} as the model system, the noise process $\boldsymbol{\varepsilon}$ in place of the observation process, and a different cost function (product instead of loss), Theorem \ref{Thm:KappaVar} is a consequence of \cite[Theorem 1.4]{McGoffNobel}. 
\end{proof}

\subsection{Connecting entropy and mean width} \label{Sect:EntropyAndMeanWidth}

In this section we investigate connections between the notions of entropy and mean width for continuous families of dynamical models. 
We begin by proving that a family with zero entropy must have zero mean width relative to centered i.i.d. processes. Our proof relies on a foundational result of Furstenberg concerning joinings, stated below as Theorem \ref{Thm:Disjointness}.

Let $\mathbf{U}$ be a stationary stochastic process taking values in a separable completely metrizable space $\mathcal{U}$. Let $\pi$ be a finite Borel partition of $\mathcal{U}$. For $n \geq 1$, and $A_0^{n-1} = (A_0,\dots,A_{n-1}) \in \pi^n$, let
\begin{equation*}
p(A_0^{n-1}) = \mathbb{P} \bigl( U_{0} \in A_0, \dots, U_{n-1} \in A_{n-1} \bigr),
\end{equation*}
and consider
\begin{equation*}
H_n(\mathbf{U},\pi) = - \sum_{A_0^{n-1} \in \pi^n}  p(A_0^{n-1}) \log p(A_0^{n-1}),
\end{equation*}
with the convention that $0 \cdot \log 0 = 0$.
By subadditivity, we may take the limit as $n$ tends to infinity:
\begin{equation*}
h(\mathbf{U},\pi) = \lim_n \frac{1}{n} H_n(\mathbf{U},\pi).
\end{equation*}
Then taking the supremum over all finite Borel partitions of $\mathcal{U}$ gives the entropy of the process $\mathbf{U}$:
\begin{equation*}
h(\mathbf{U}) = \sup_{\pi} h(\mathbf{U},\pi).
\end{equation*}

In proving Theorem \ref{Thm:Furstenberg}, we will rely on the following result of Furstenberg.

\begin{CorThm} \cite[Theorem I.2]{Furstenberg1967} \label{Thm:Disjointness}
If $h(\mathbf{U})=0$ and $\mathbf{V}$ is i.i.d., then the only joining of $\mathbf{U}$ and $\mathbf{V}$ is the independent joining.
\end{CorThm}

\vspace{2mm}


\begin{PfofThmFurstenberg}

Let $\mathbf{U} \in \mathcal{Q}_{\mathcal{D}}$. Since $h(\mathcal{D}) = 0$, Lemma \ref{Lemma:VarPrin} implies that $h(\mathbf{U})=0$. Then by the result of Furstenberg (Theorem \ref{Thm:Disjointness} above), the only joining of $\mathbf{U}$ with $\boldsymbol{\varepsilon}$ is the independent joining. Thus, for any joining $[\mathbf{U},\boldsymbol{\varepsilon}]$ with $\mathbf{U} \in \mathcal{Q}_{\mathcal{D}}$, we have
\begin{equation*}
\mathbb{E} \bigl[ U_0 \cdot \varepsilon_0 \bigr] =  \mathbb{E} \bigl[ U_0 \bigr] \cdot \mathbb{E} \bigl[ \varepsilon_0 \bigr] = 0.
\end{equation*}
Then by Theorem \ref{Thm:KappaVar},
\begin{equation*}
\kappa(\mathcal{D} : \boldsymbol{\varepsilon}) = \sup_{\mathbf{U} \in \mathcal{Q}_{\mathcal{D}}} \sup_{\cJ(\mathbf{U}, \boldsymbol{\varepsilon})} \mathbb{E} \bigl[ U_0 \cdot \varepsilon_0 \bigr]  = 0.
\end{equation*}
%
%

Now suppose that $h(\mathcal{D}) >0$ and that $(\varepsilon_k)_{k \geq 0}$ are i.i.d.\ standard normal random variables.
For each $n \geq 1$, an application of Sudakov's lower bound \cite{Sudakov1973} (see also \cite[Theorem 6.1]{Ledoux1996}) yields
\begin{align*}
 \frac{1}{n} \kappa_n(\mathcal{D} : \boldsymbol{\varepsilon}) & \geq  
   \sup_{\delta>0} \frac{\delta}{6} \biggl( \frac{1}{n} \log N(\mathcal{U}_{\mathcal{D}},\delta,d_{n,2}) \biggr)^{1/2}.
\end{align*}
By Theorem \ref{Thm:EntropyEquality} and the positive entropy hypothesis, we have $h_{2}(\mathcal{U}_{\mathcal{D}}) = h(\mathcal{D}) >0$. Therefore there exists $\delta>0$ such that $h_{2}(\mathcal{U}_{\mathcal{D}},\delta)>0$. With this choice of $\delta$, we take $n$ to infinity in the previous display and obtain
\begin{equation*}
 \kappa_{G}(\mathcal{D}) \geq \frac{\delta}{6} \bigl( h_{2}(\mathcal{U}_{\mathcal{D}},\delta) \bigr)^{1/2} > 0.
\end{equation*}
\end{PfofThmFurstenberg}

\section{Signal plus noise (proofs)} \label{Sect:SNproofs}

This section addresses empirical risk minimization for families dynamical models in the signal plus noise setting.
Recall that joinings were defined for $\R$-valued processes in Definition \ref{Defn:Joinings}. 

The following result is an extension of Furstenberg's result (stated above as Theorem \ref{Thm:Disjointness}), which we use to show that minimum risk estimates decouple the signal from the noise in the low complexity setting. 
For a proof, see Appendix \ref{Sect:OurDisjointnessProof}.

\begin{thm} \label{Thm:OurDisjointness}
Suppose that $\mathbf{U}$ is a zero entropy process, $\mathbf{V}$ is a stationary ergodic process, and $\mathbf{W}$ is an i.i.d. process. 
Suppose that $[\mathbf{U},\mathbf{V},\mathbf{W}]$ is a joining of these three processes  such that $\mathbf{V}$ and $\mathbf{W}$ are independent. Then $[\mathbf{U},\mathbf{V}]$ is independent of $\mathbf{W}$.
\end{thm}


Our proofs also require the concept of the relatively independent joining, which results from a standard construction in ergodic theory \cite[p. 126]{Glasner2003}. For notation, if $\mathbf{U}$ is a stationary processes taking values on $\mathcal{U}$ and $f : \mathcal{U} \to \mathcal{W}$ is a measurable map, then we let $f(\mathbf{U})$ denote the $\mathcal{W}$-valued process $\{f(U_k)\}_{k \geq 0}$, and we say that $f$ maps $\mathbf{U}$ onto $\mathbf{W}$ whenever $f(\mathbf{U})$ has the same distribution as $\mathbf{W}$.
\begin{CorThm}[Relatively Independent Joining] \label{Prop:RIJ}
Suppose $\mathbf{U}$, $\mathbf{V}$, and $\mathbf{W}$ are stationary processes taking values on separable completely metrizable spaces. If there are Borel measurable maps $f$ and $g$ such that $f(\mathbf{U})$ and $g(\mathbf{V})$ each have the same distribution as $\mathbf{W}$, then there is a joining $[\mathbf{U}, \mathbf{V}]$ of $\mathbf{U}$ and $\mathbf{V}$ such that $f(\mathbf{U}) = g( \mathbf{V})$ almost surely.
\end{CorThm}
The joining whose existence is asserted by Theorem \ref{Prop:RIJ} will be called the relatively independent joining of $\mathbf{U}$ and $\mathbf{V}$ (relative to $\mathbf{W}$). It will be used several times in the following proofs. 

\vspace{2mm}

\begin{PfofThmGenObsModelConsistency}
 
 By Theorem \ref{Thm:GEN} any sequence of minimal $\ell$-risk parameters converges almost surely to $\Theta_{\ell}(\mathbf{Y})$. We will complete the proof by showing that $\Theta_{\ell}(\mathbf{Y}) = \Theta_L(\mathbf{V})$.
 
 Let $\mathbf{U}$ be in $\mathcal{Q}_{\mathcal{D}}$. Let $[\mathbf{U},\mathbf{Y}]$ be a joining of $\mathbf{U}$ and $\mathbf{Y}$ such that $\mathbb{E} \bigl[ \ell(U_0,Y_0) \bigr] = \gamma_{\ell}(\mathbf{U},\mathbf{Y})$. Also,  let $[\mathbf{V},\boldsymbol{\varepsilon}]$ be the independent joining of $\mathbf{V}$ and $\boldsymbol{\varepsilon}$. Since $\mathbf{Y}$ has the same distribution as $\mathbf{V}+\epsilon$, we may apply Theorem \ref{Prop:RIJ} and let $[\mathbf{U},\mathbf{Y},\mathbf{V},\boldsymbol{\varepsilon}]$ be the relatively independent joining of $[\mathbf{U},\mathbf{Y}]$ and $[\mathbf{V},\boldsymbol{\varepsilon}]$ such that $\mathbf{Y} = \mathbf{V} + \boldsymbol{\varepsilon}$ almost surely. Projecting $[\mathbf{U},\mathbf{Y},\mathbf{V},\boldsymbol{\varepsilon}]$ onto the first, third, and fourth coordinates, we obtain a joining $[\mathbf{U},\mathbf{V},\boldsymbol{\varepsilon}]$ satisfying the conditions of Theorem \ref{Thm:OurDisjointness}. Applying that theorem, we obtain that the $[\mathbf{U},\mathbf{V}]$ is independent of $\boldsymbol{\varepsilon}$. 
 By conditioning on $[\mathbf{U},\mathbf{V}]$ and using that these variables are jointly independent of $\boldsymbol{\varepsilon}$, we see that
 \begin{align*}
 \mathbb{E} \bigl[ \ell(U_0,Y_0) \bigr]  = \mathbb{E} \bigl[ \ell(U_0, V_0 + \varepsilon_0) \bigr] 
  & = \mathbb{E} \Bigl[  \mathbb{E} \bigl[ \ell(u,v +\varepsilon_0)  \mid U_0 = u, V_0 = v \bigr] \Bigr] \\
   & = \mathbb{E} \Bigl[  L(U_0,V_0)  \Bigr] 
   \geq \gamma_L(\mathbf{U},\mathbf{V}).
 \end{align*}
 Hence we have shown that $\gamma_{\ell}(\mathbf{U}, \mathbf{Y}) \geq \gamma_{L}(\mathbf{U}, \mathbf{V})$.
 
 Now let $[\mathbf{U},\mathbf{V}]$ be a joining such that $\mathbb{E} \bigl[ L(U_0,V_0) \bigr] = \gamma_L(\mathbf{U},\mathbf{V})$. Let $[\mathbf{U},\mathbf{V},\boldsymbol{\varepsilon}]$ be the independent joining of $[\mathbf{U},\mathbf{V}]$ with $\boldsymbol{\varepsilon}$, and let $\mathbf{Y} = \mathbf{V} + \boldsymbol{\varepsilon}$. Then
 \begin{align*}
 \mathbb{E} \bigl[  L(U_0,V_0) \bigr] 
 & = \mathbb{E} \Bigl[  \mathbb{E} \bigl[ \ell(u,v +\varepsilon_0)  \mid U_0 = u, V_0 = v \bigr] \Bigr] \\
 & = \mathbb{E} \bigl[ \ell(U_0, V_0 + \varepsilon_0) \bigr] =  \mathbb{E} \bigl[ \ell(V_0,Y_0) \bigr]  \geq \gamma_{\ell}(\mathbf{U},\mathbf{Y}).
 \end{align*}
 Thus $\gamma_L(\mathbf{U}, \mathbf{V}) \geq \gamma_{\ell}(\mathbf{U}, \mathbf{Y})$. Combining this inequality with the reverse inequality, which we established above, yields that $\gamma_L(\mathbf{U}, \mathbf{V}) = \gamma_{\ell}(\mathbf{U}, \mathbf{Y})$. As $\mathbf{U} \in \mathcal{Q}_{\mathcal{D}}$ was arbitrary, we obtain that $\Theta_{\ell}(\mathbf{Y}) = \Theta_L(\mathbf{V})$, which finishes the proof.
   \end{PfofThmGenObsModelConsistency}

\vspace{2mm}

\begin{PfofThm_PropSPEC}

 By Theorem \ref{Thm:GenObsModelConsistency} any sequence of minimum risk parameters converges almost surely to $\Theta_L(\mathbf{V})$. 
 Assume for (1) that $\E \, \varepsilon_0 = 0$ and that $\ell(u,v) = D_F(v,u)$ is a Bregman divergence of a continuously differentiable convex
function $F: \R \to \R$, namely
\[
\ell(u,v) \ := \ F(v) - F(u) - (v-u) F'(u) .
\]
Since $\E (\varepsilon) = 0$, we have
\begin{align*}
L(u,v) 
= 
\E \ell(u, v + \varepsilon_0) & = \E F(v + \varepsilon_0) - F(u) - (v-u) F'(u) \\
& = \ell(u,v) + G(v),
\end{align*}
where $G(v) = \E F(v + \varepsilon_0) - F(v)$ depends only on $v$ and the distribution of $\varepsilon_0$ and is non-negative (since
$F$ is convex).  Thus, for any $\theta \in \Theta$ and any $\mathbf{U} \in \mathcal{Q}_\theta$,
\begin{eqnarray*}
\gamma_L(\mathbf{U},\mathbf{V})
& = & 
\inf_{\mathcal{J}(\mathbf{U},\mathbf{V})} \E L(U_0, V_0) \\
& = &
\inf_{\mathcal{J}(\mathbf{U},\mathbf{V})} \Bigl\{ \E \ell(U_0, V_0) + \E G(V_0) \Bigr\} \\
& = & 
\gamma_\ell (\mathbf{U},\mathbf{V}) + \E G(V_0) .
\end{eqnarray*}
It follows that $\Theta_L(\mathbf{V}) = \Theta_{\ell}(\mathbf{V})$, as desired for (1).

Now assume for (2) that $\mathbf{V}$ is an ergodic process in $\mathcal{Q}_{\theta_0}$ and 
$\mathbb{E} \, \ell(x,y+\varepsilon_0) \geq \mathbb{E} \, \ell(0,\varepsilon_0)$ for all $x,y \in \R$,
with equality if and only if $x = y$. Let $\Theta_1 = \{ \theta \in \Theta : \mathbf{V} \in \mathcal{Q}_{\theta} \}$. We will finish the proof by showing that $\Theta_L(\mathbf{V}) =  \Theta_1$.
 
 Let $\mathbf{U} \in \mathcal{Q}_{\mathcal{D}}$, and let $[\mathbf{U},\mathbf{V}]$ be a joining of these two processes. By the hypothesis concerning $\boldsymbol{\varepsilon}$ and $\ell$, we have
 \begin{align} \label{Eqn:Banff}
 \mathbb{E} \bigl[ L(U_0,V_0) \bigr] & = \mathbb{E} \Bigl[ \mathbb{E} \bigl[ \ell(u, v+\epsilon_0)  \mid  U_0 = u, V_0 = v \bigr] \biggr] \\
 & \geq \mathbb{E} \bigl[ \ell(0,\epsilon_0) \bigr],
 \end{align}
 with equality if and only if $U_0 = V_0$ almost surely. Since $[\mathbf{U},\mathbf{V}]$ is a joining (and in particular is stationary), we observe that $U_0 = V_0$ almost surely if and only if $\mathbf{U} = \mathbf{V}$ almost surely. Thus, we have shown that $\gamma_L(\mathbf{U},\mathbf{V}) \geq  \mathbb{E} \bigl[ \ell(0,\epsilon_0) \bigr]$, with equality if and only if $\mathbf{U} = \mathbf{V}$. Therefore the set of $\theta$ minimizing the quantity $\min_{\mathbf{U} \in \mathcal{Q}_{\theta}} \gamma_L( \mathbf{U}, \mathbf{V})$ is exactly the set of $\theta$ such that $\mathbf{V} \in \mathcal{Q}_{\theta}$. Hence, $\Theta_L(\mathbf{V}) = \Theta_1$, which finishes the proof.
\end{PfofThm_PropSPEC}

\subsection{Least squares estimation}

Here we provide short proofs of our results pertaining to least squares estimation. It is possible to give somewhat more direct proofs of these results (avoiding Theorem \ref{Thm:GEN}, for example), but given the tools that we have already established, we present the most efficient proofs of which we are aware.

\vspace{2mm}

\begin{PfofThm_LS}

By Theorem \ref{Thm:GEN}, any sequence of least squares parameters converges almost surely to $\Theta_{\ell}(\mathbf{Y})$. We will complete the proof by showing that $\Theta_{\ell}(\mathbf{Y}) = \argmin_{\theta} \min_{\mathbf{U} \in \mathcal{Q}_{\theta}} \gamma_2(\mathbf{U},\mathbf{V})$.

Let $\theta \in \Theta$.
Let $\mathbf{U}$ be in $\mathcal{Q}_{\theta}$, and let $[\mathbf{U},\mathbf{Y}]$ be a joining of these processes. Let $[\mathbf{V},\boldsymbol{\varepsilon}]$ be the independent joining, and (using Theorem \ref{Prop:RIJ}) let $[\mathbf{U},\mathbf{Y},\mathbf{V},\boldsymbol{\varepsilon}]$ be the relatively independent joining of $[\mathbf{U},\mathbf{Y}]$ with $[\mathbf{V},\boldsymbol{\varepsilon}]$ such that $\mathbf{Y} = \mathbf{V} + \boldsymbol{\varepsilon}$ almost surely.
Then 
\begin{align}
\begin{split} \label{Eqn:Basilisk}
\mathbb{E} \Bigl[ \bigl| U_0 - Y_0 \bigr|^2 \Bigr] & = \mathbb{E} \Bigl[ \bigl| U_0 - (V_0 + \varepsilon_0) \bigr|^2 \Bigr] \\
& = \mathbb{E} \Bigl[ \bigl| U_0 - V_0 \bigr|^2 \Bigr] - 2 \mathbb{E} \bigl[ (U_0 - V_0) \cdot \varepsilon_0 \bigr] + \mathbb{E} \Bigl[ \varepsilon_0^2 \Bigr].
\end{split}
\end{align}
Since $\kappa(\mathcal{D} : \boldsymbol{\varepsilon}) = 0$, we have $\mathbb{E} [ U_0 \cdot \varepsilon_0 ] = 0$ by Theorem \ref{Thm:KappaVar}. Also, since $V_0$ is independent of $\varepsilon_0$ and $\varepsilon_0$ has zero mean, we have $\mathbb{E}[ V_0 \cdot \varepsilon_0 ] =  0$. Applying these facts in (\ref{Eqn:Basilisk}), we see that
\begin{equation} \label{Eqn:Narwhal}
\mathbb{E} \Bigl[ \bigl| U_0 - Y_0 \bigr|^2 \Bigr] =  \mathbb{E} \Bigl[ \bigl| U_0 - V_0 \bigr|^2 \Bigr]  + \mathbb{E} \Bigl[\varepsilon_0^2 \Bigr].
\end{equation}
Since $\mathbb{E}[\epsilon_0^2]$ is a constant that depends only on $\epsilon_0$ (and not on the joining), we conclude that
\begin{equation} \label{Eqn:Stingray}
\min_{\mathbf{U} \in \mathcal{Q}_{\theta}} \gamma_2( \mathbf{U}, \mathbf{Y}) \geq \min_{\mathbf{U} \in \mathcal{Q}_{\theta}} \gamma_2( \mathbf{U},  \mathbf{V}) + \mathbb{E}[\epsilon_0^2].
\end{equation}

Now let $\mathbf{U} \in \mathcal{Q}_{\theta}$ and $[\mathbf{U},\mathbf{V}]$ be a joining such that $\mathbb{E} [ |U_0-V_0|^2] = \min_{\mathbf{U} \in \mathcal{Q}_{\theta}} \gamma_2(\mathbf{U}, \mathbf{V}) $. Let $\boldsymbol{\varepsilon}$ be independent of $[\mathbf{U},\mathbf{V}]$, and let $\mathbf{Y} = \mathbf{V}+\boldsymbol{\varepsilon}$. Then $(\mathbf{U},\mathbf{Y})$ is a joining of these processes such that
\begin{equation} \label{Eqn:Louise}
\mathbb{E} \Bigl[ \bigl| U_0 - Y_0 \bigr|^2 \Bigr] =  \mathbb{E} \Bigl[ \bigl| U_0 - V_0 \bigr|^2 \Bigr]  + \mathbb{E} \Bigl[\varepsilon_0^2 \Bigr] =  \min_{\mathbf{U} \in \mathcal{Q}_{\theta}} \gamma_2(\mathbf{U}, \mathbf{V})  +  \mathbb{E} \Bigl[\varepsilon_0^2 \Bigr].
\end{equation}
Combining (\ref{Eqn:Stingray}) and (\ref{Eqn:Louise}), we obtain that $\min_{\mathbf{U} \in \mathcal{Q}_{\theta}} \gamma_2(\mathbf{U}, \mathbf{Y})$ is equal to $\min_{\mathbf{U} \in \mathcal{Q}_{\theta}} \gamma_2(\mathbf{U}, \mathbf{V})  +  \mathbb{E} \bigl[\varepsilon_0^2 \bigr]$. Then by minimizing over $\theta$, we see that $\Theta_{2}(\mathbf{Y}) = \argmin_{\theta} \min_{\mathbf{U} \in \mathcal{Q}_{\theta}} \gamma_2(\mathbf{U}, \mathbf{V})$, as was to be shown.
\end{PfofThm_LS}

\vspace{2mm}

\begin{PfofCorRd}
 
 The hypothesis that each $T_{\theta}$ has zero entropy can be easily seen to be equivalent in this context to the statement that $h(\mathcal{D})=0$.
 Then by Theorem \ref{Thm:SPEC} any sequence of least squares parameters converges almost surely to the set $\{\theta \in \Theta : \mathbf{V} \in \mathcal{Q}_{\theta}\}$. Now suppose $\mathbf{V} \in \mathcal{Q}_{\theta}$. Then there exists a measure $\mu_0 \in \mathcal{M}(\cX,T_{\theta})$ such that if $\mathbf{U} = (T_{\theta}^k(X))_{k \geq 0}$, with $X$ distributed according to $\mu_0$, then $\mathbf{U}$ has the same distribution as $\mathbf{V}$. Hence $X$ has the same distribution as $V_0$, which is given by $\mu$, and therefore $\mu = \mu_0$. Furthermore, $(X,T_{\theta}(X))$ must have the same distribution as $(V_0,V_1)$, which implies that $T_{\theta}(x) = T_{\theta^*}(x)$ for $\mu$ almost every $x$. We have thus shown that $\{\theta \in \Theta : \mathbf{V} \in \mathcal{Q}_{\theta}\} \subset \{ \theta : \mu(T_{\theta} = T_{\theta^*}) = 0 \}$. The reverse inclusion is obvious.
\end{PfofCorRd}


%
%
\section{Negative results (proofs)} \label{Sect:NegativeExamples}

In this section, we show that least squares estimation can be almost surely \textit{inconsistent} in a properly specified setting when $h(\mathcal{D}) >0$ (or equivalently, $\kappa_G(\mathcal{D})>0$). 
We begin with a proof of Proposition \ref{Prop:Negative}. 

\vspace{2mm}

\begin{PfofPropNegative}

For the existence of a family $\mathcal{D}$ satisfying the hypotheses of the proposition, see Example \ref{Example:KappaPositive}. Suppose the hypotheses hold for a family $\mathcal{D}$ as in the statement of the proposition. In particular, fix an ergodic $\mathbf{V} \in \mathcal{Q}_{\theta_0} \setminus \cup_{\theta \neq \theta_0} \mathcal{Q}_{\theta}$.

By compactness and continuity, there exists $K$ such that $|f_{\theta}(x)| \leq K$ for all $x \in \mathcal{X}$ and $\theta \in \Theta$. Since $h(\mathcal{D}) >0$, Theorem \ref{Thm:Furstenberg} yields that $\kappa_G(\mathcal{D}) >0$. Let $\sigma_0 = K^2/(2 \kappa_G(\mathcal{D}))$. Let $\boldsymbol{\varepsilon}$ be an i.i.d. process of standard normal random variables, and let $\boldsymbol{\varepsilon}(\sigma) = \sigma \boldsymbol{\varepsilon}$. Now for $\sigma > \sigma_0$, let $\mathbf{Y} = \mathbf{V}+\boldsymbol{\varepsilon}(\sigma)$.

Let $\mathcal{D}_0 = \{ (T_{\theta_0},f_{\theta_0}) \}$ be the family containing only the dynamical model corresponding to the parameter $\theta_0$. By Lemma \ref{Lemma:VarPrin} and our hypothesis on $\theta_0$, we have
\begin{equation*}
 h(\mathcal{D}_0) = \sup_{\mathbf{U} \in \mathcal{Q}_{\theta_0}} h(\mathbf{U}) = 0.
\end{equation*}
Then by (\ref{Eqn:Narwhal}), for any joining $[\mathbf{U},\mathbf{Y}]$ with $\mathbf{U} \in \mathcal{Q}_{\theta_0}$, we obtain
\begin{equation} \label{Eqn:Blue}
 \mathbb{E} \biggl[ \bigl| U_0 - Y_0 \bigr|^2 \biggr] = \mathbb{E} \biggl[ \bigl| U_0 - V_0 \bigr|^2 \biggr] + \mathbb{E}\Bigl[ \varepsilon_0^2 \Bigr] \geq \mathbb{E}\Bigl[ \varepsilon_0^2 \Bigr] = \sigma^2.
\end{equation}
Hence, for any $\mathbf{U} \in \mathcal{Q}_{\theta_0}$, we have $\gamma_{2}(\mathbf{U},\mathbf{Y}) \geq \sigma^2$.

Let us now show that there exists $\mathbf{U} \in \mathcal{Q}_{\mathcal{D}}$ such that $\gamma_2(\mathbf{U}, \mathbf{Y}) < \sigma^2$. As we have already established above, $\kappa_G(\mathcal{D}) >0$. Then by Theorem \ref{Thm:KappaVar}, there exists a process $\mathbf{U} \in \mathcal{Q}_{\mathcal{D}}$ and a joining $[\mathbf{U},\boldsymbol{\varepsilon}]$ such that 
\begin{equation*}
 \mathbb{E} \Bigl[ U_0 \cdot \varepsilon_0(\sigma) \Bigr] = \sigma \mathbb{E} \Bigl[ U_0 \cdot \varepsilon_0 \Bigr] = \sigma \kappa_{G}(\mathcal{D}).
\end{equation*}
Let $[\mathbf{U}, \mathbf{V},\boldsymbol{\varepsilon}]$ be the independent joining of $\mathbf{V}$ with $[\mathbf{U},\boldsymbol{\varepsilon}]$, and let $\mathbf{Y} = \mathbf{V}+\boldsymbol{\varepsilon}$. Then we have
 \begin{align*}
  \gamma_2(\mathbf{U},\mathbf{Y}) & \leq \mathbb{E} \biggl[ \bigl|U_0 - (V_0+\varepsilon_0) \bigr|^2 \biggr] \\
  & =  \mathbb{E} \Bigl[ \bigl| U_0 - V_0 \bigr|^2 \Bigr] - 2 \mathbb{E} \bigl[ (U_0 - V_0) \cdot \varepsilon_0 \bigr] + \mathbb{E} \Bigl[ \varepsilon_0^2 \Bigr] \\
  & \leq K^2 - 2 \sigma \kappa_{G}(\mathcal{D}) + \sigma^2 < \sigma^2,
 \end{align*}
 where the last inequality follows from our choice of $\sigma > \sigma_0$. 
 Combining this inequality with (\ref{Eqn:Blue}), we conclude that $\theta_0$ is not contained in $\Theta_2(\mathbf{Y})$, which is the nonempty compact limit set of least squares estimates by Theorem \ref{Thm:GEN}.
\end{PfofPropNegative}

\vspace{2mm} 

The following example establishes the existence of a family $\mathcal{D}$ of dynamical models satisfying the hypotheses of Proposition \ref{Prop:Negative}. In fact, the parameter set $\Theta$ in this example only contains two values.

\begin{example} \label{Example:KappaPositive}
 Consider the state space $\cX = [0,1]$ with two transformations: let $T_0$ be the identity map on $[0,1]$, and let $T_1$ be the fully chaotic logistic map, given by $T_1(x) = 4x(1-x)$. We let $f_0 = f_1$ be the identity map on $[0,1]$. This information fully determines $\mathcal{D}$.
 
 As $T_0$ is the identity, we clearly have that $h(\mathbf{U}) = 0$ for all $\mathbf{U} \in \mathcal{Q}_{\theta_0}$. Also, we may let $\mathbf{V}$ be the constant process $V_k = 1/2$ for all $k$. Note that $\mathbf{V} \notin \mathcal{Q}_{\theta_1}$. Finally, it is well known that $T_1$ has positive entropy. We have thus verified the hypotheses of Proposition \ref{Prop:Negative}. 
\end{example}

\bibliographystyle{plain}
\bibliography{Optimization_Applications_refs}


\appendix

\section{Additional proofs} \label{Sect:AdditionalProofs}

\subsection{Proof of Lemma \ref{Lemma:VarPrin}}

The well-known variational principle for topological dynamical systems \cite[p. 190]{Walters1982} equates the topological entropy $h_{\mathrm{top}}(T)$ of a continuous transformation $T: \cX \to \cX$ of a compact metrizable space with the supremum of the measure-theoretic entropies $h_T(\mu)$ over $\mu \in \mathcal{M}(\cX,T)$. Here we show how that variational principle may be applied in our setting to obtain Lemma \ref{Lemma:VarPrin}. In the following lemma and its proof, we use subscripts to distinguish the different types of entropy under consideration. Here we write $h_{\mathrm{fam}}(\mathcal{D})$ to denote the quantity $h(\mathcal{D})$ defined in Section \ref{Sect:EntropyDef}, and we write $h_{\mathrm{proc}}(\mathbf{U})$ to denote the quantity $h(\mathbf{U})$ defined in Section \ref{Sect:EntropyAndMeanWidth}.
\begin{lemma} 
For any continuous family $\mathcal{D}$ of dynamical models,
\begin{equation*}
h_{\mathrm{fam}}(\mathcal{D}) = \sup_{\mathbf{U} \in \mathcal{Q}_{\mathcal{D}}} h_{\mathrm{proc}}(\mathbf{U}).
\end{equation*}
\end{lemma}
\begin{proof}
Consider the dynamical system $(\mathcal{Z},T)$ associated to $\mathcal{D}$ in Section \ref{Sect:Connections}, and let $\mathcal{U}$ be the set of sequences of the form $(f_{\theta} \circ T_{\theta}^k(x))_{k \geq 0}$ for any $\theta \in \Theta$ and $x \in \cX$. By Lemma \ref{Lemma:BigShift}, the standard variational principle for topological dynamical systems may be applied, and therefore
\begin{equation*} 
h_{\mathrm{top}}(T) = \sup_{\mu \in \mathcal{M}(\mathcal{Z},T)} h_T(\mu),
\end{equation*}
where the supremum may be taken over ergodic measures $\mu$.

For $\theta$ in $\Theta$, let  $\mathcal{Z}_{\theta} = \proj_{\Theta}^{-1}(\theta)$, which is closed and invariant under $T$. Also, let $T|_{\theta}$ denote the restriction of $T$ to $\mathcal{Z}_{\theta}$.

For any scale $\delta >0$, we have $N(\mathcal{U},\delta,d_{n,\infty}) \leq N(\mathcal{Z},\delta,d_{n,\infty})$, and therefore $h_{\mathrm{fam}}(\mathcal{D}) \leq h_{\mathrm{top}}(T)$. Similarly, we have that $h_{\mathrm{top}}(T|_{\theta}) \leq h_{\mathrm{fam}}(\mathcal{D})$.

By Lemma \ref{Lemma:BigShift}, for any ergodic $\mu \in  \mathcal{M}(\mathcal{Z},T)$, there exists $\theta$ such that $\mu = \delta_{\theta} \otimes \nu$, where $\nu$ is the distribution of a process $\mathbf{U} \in \mathcal{Q}_{\theta}$.  Then $h_{T}(\mu) = h_{\mathrm{proc}}(\mathbf{U})$, and we may conclude that
\begin{equation*}
\sup_{\mu \in \mathcal{M}(\mathcal{Z},T)} h_T(\mu) \leq \sup_{\theta} \sup_{\mathbf{U} \in \mathcal{Q}_{\theta}} h_{\mathrm{proc}}(\mathbf{U}).
\end{equation*}
Furthermore, for any $\theta \in \Theta$, an application of the standard variational principle to the system 
$(\mathcal{Z}_{\theta}, T|_{\theta})$ gives that
\begin{equation*}
h_{\mathrm{top}}(T|_{\theta}) = \sup_{ \mu \in \mathcal{M}(\mathcal{Z}_{\theta},T|_{\theta})} h_{T|_{\theta}}(\mu) = \sup_{\mathbf{U} \in \mathcal{Q}_{\theta}} h_{\mathrm{proc}}(\mathbf{U}),
\end{equation*}
where the second equality follows from the equivalence between processes in $\mathcal{Q}_{\theta}$ and measures in  $\mathcal{M}(\mathcal{Z}_{\theta},T|_{\theta})$.
Combining all of the statements above, we see that
\begin{equation*}
 h_{\mathrm{top}}(T) = \sup_{\mu \in \mathcal{M}(\mathcal{Z},T)} h_T(\mu) \leq \sup_{\mathbf{U} \in \mathcal{Q}_{\mathcal{D}}} h_{\mathrm{proc}}(\mathbf{U}) \leq h_{\mathrm{fam}}(\mathcal{D}) \leq h_{\mathrm{top}}(T),
\end{equation*}
which concludes the proof of the lemma.
\end{proof}

\subsection{Proof of Theorem \ref{Thm:OurDisjointness}} \label{Sect:OurDisjointnessProof}

Here we provide the proof of Theorem \ref{Thm:OurDisjointness}, which is really an adaptation of Furstenberg's original proof of Theorem \ref{Thm:Disjointness}. We begin with a lemma, which reduces the proof to the case of finite-valued processes.

\begin{lemma}
Suppose Theorem \ref{Thm:OurDisjointness} holds for finite-valued processes. Then it holds in general.
\end{lemma}
\begin{proof}
Suppose $\mathbf{U}$ takes values in the measurable space $(\mathcal{U},\mathcal{F}_1)$, $\mathbf{V}$ in $(\mathcal{V},\mathcal{F}_2)$, and $\mathbf{W}$ in $(\mathcal{W},\mathcal{F}_3)$. Then the product $\sigma$-algebra on $\mathcal{U}^{\N} \times \mathcal{V}^{\N} \times \mathcal{W}^{\N}$ is generated by sets of the form $(A_1 \times \dots \times A_N) \times (B_1 \times \dots \times B_N) \times (C_1 \times \dots \times C_N)$, where $A_i \in \mathcal{F}_1$, $B_i \in \mathcal{F}_2$, and $C_i \in \mathcal{F}_3$. Fix $N$ and such $A_1,\dots,A_N$, $B_1,\dots, B_N$, and $C_1,\dots,C_N$. Let $\pi_1 = \bigvee_i \{A_i,A_i^c\}$, $\pi_2 = \bigvee_i \{ B_i, B_i^c\}$, and $\pi_3 = \bigvee_i \{C_i,C_i^c\}$, where $\bigvee$ denotes the join. Now consider the finite-valued processes $\pi_1(\mathbf{U})$, $\pi_2(\mathbf{V})$, and $\pi_3(\mathbf{W})$. Note that $\pi_1(\mathbf{U})$ has zero entropy, $\pi_2(\mathbf{V})$ is stationary and ergodic, and $\pi_3(\mathbf{W})$ is i.i.d. Furthermore, $\pi_2(\mathbf{V})$ is independent of $\pi_3(\mathbf{W})$. Hence, by Theorem \ref{Thm:OurDisjointness} for finite-valued processes gives that the joint process $(\pi_1(\mathbf{U}),\pi_2(\mathbf{V}))$ is independent of $\pi_3(\mathbf{W})$. Then
\begin{align*}
& \mathbb{P} \Bigl(  \mathbf{U} \in A_1 \times \dots \times A_N, \mathbf{V} \in B_1 \times \dots \times B_N, \mathbf{W} \in C_1 \times \dots \times C_N \Bigr) \\
& = \mathbb{P} \Bigl( \mathbf{U} \in  A_1 \times \dots \times A_N, \mathbf{V} \in B_1 \times \dots \times B_N \Bigr) \mathbb{P} \Bigl( \mathbf{W} \in C_1 \times \dots \times C_N \Bigr).
\end{align*}
Since $N$ and these sets were arbitrary, we see that $[\mathbf{U},\mathbf{V}]$ is independent of $\mathbf{W}$.
\end{proof}

Now, for the proof of Theorem \ref{Thm:OurDisjointness}, assume that all three processes are finite valued. In the rest of this proof, we use standard properties of entropy and stationary ergodic processes  (see \cite{Cover2012,Glasner2003}). First, notice that
\begin{align*}
h(\mathbf{V}) + H(W_1) & = h( [\mathbf{V},\mathbf{W}] ) \\
 & \leq h( [\mathbf{U},\mathbf{V},\mathbf{W}] )  \\
 & \leq h(\mathbf{U}) + h( [\mathbf{V},\mathbf{W}] )  \\
 & = h(\mathbf{V}) + H(W_1),
\end{align*}
where the two equalities rely on our hypotheses: $h(\mathbf{U}) = 0$, $[\mathbf{V},\mathbf{W}]$ is the independent joining, and $\mathbf{W}$ is i.i.d. Hence $h(\mathbf{V})+H(W_1) = h( [\mathbf{U},\mathbf{V},\mathbf{W}] ) $.
Then by conditioning and using the information cocycle equation \cite[p. 255]{Glasner2003}, we have
\begin{align*}
h(\mathbf{V}) & + H(W_1) \\
& = h([\mathbf{U},\mathbf{V}, \mathbf{W}]) \\
& = H( U_1, V_1, W_1 \mid U_2, V_2, W_2 , \dots) \\
& = H(U_1,V_1 \mid U_2, V_2, W_2 , \dots) + H(W_1 \mid U_1, V_1, U_2, V_2, W_2, \dots) \\
& \leq H(U_1 \mid U_2, V_2, W_2 , \dots) + H( V_1 \mid U_2, V_2, W_2 , \dots) \\
& \quad \quad \quad \quad \quad \quad \quad \quad \quad \quad \quad \quad + H(W_1 \mid U_1, V_1, U_2, V_2, W_2, \dots) \\
& \leq h(\mathbf{U}) + h(\mathbf{V}) + H(W_1 \mid U_1, V_1, U_2, V_2, W_2, \dots) \\
& = h(\mathbf{V}) + H(W_1 \mid U_1, V_1, U_2, V_2, W_2, \dots) \\
& \leq h(\mathbf{V}) + H(W_1).
\end{align*}
Since $\mathbf{V}$ is finite-valued, $h(\mathbf{V})$ is finite. Then the previous display yields
\begin{equation*}
H(W_1 \mid U_1, V_1, U_2, V_2, W_2, \dots) = H(W_1),
\end{equation*}
which implies that $W_1$ is independent of the $\sigma$-algebra generated by the random variables $\{U_1,V_1,U_2,V_2,\dots\}$. For arbitrary $r \in \N$, repetition of this argument for the variables $(W_1,\dots,W_r)$, $(W_{r+1},\dots,W_{2r})$, $\dots$, shows that the entire process $\mathbf{W}$ is independent of $\{U_1,V_1,U_2,V_2,\dots\}$. This concludes the proof of the theorem.

\end{document}